\documentclass[letterpaper,12pt,reqno]{amsart}
\usepackage{amssymb}

\usepackage{stmaryrd}
\usepackage{amsmath,color,amscd}
\usepackage{mathrsfs}
\usepackage{eucal}
\usepackage{amsfonts}
\usepackage[all]{xy}
\usepackage{pstricks}
\usepackage{pstricks,pst-node}

\DeclareFontFamily{OT1}{pzc}{}
\DeclareFontShape{OT1}{pzc}{m}{it}{<-> [1.15] rpzcmi}{}
\DeclareMathAlphabet{\mathzc}{OT1}{pzc}{m}{it}

\newcommand{\Z}{\mathbb{Z}}
\newcommand{\0}{\bar 0}

\def\la{{\lambda}}
\def\al{{\alpha}}

\def\ttM{{\mathtt M}}
\def\ttC{{\mathtt C}}
\def\tte{{\mathtt e}}
\def\ttf{{\mathtt f}}
\def\ttk{{\mathtt k}}

\def\sfe{{\textsc{e}}}
\def\sff{{\textsc{f}}}
\def\sfk{{\textsc{k}}}
\def\sfF{{\mathsf F}}

\newtheorem{theorem}{Theorem}[section]
\newtheorem{lemma}[theorem]{Lemma}
\newtheorem{proposition}[theorem]{Proposition}
\newtheorem{corollary}[theorem]{Corollary}

\theoremstyle{definition}
\newtheorem{definition}[theorem]{Definition}

\newtheorem{remark}[theorem]{Remark}
\newtheorem{remarks}[theorem]{Remarks}

\numberwithin{equation}{theorem}

\def\xy{{{\text{\tiny $[$}}xy{\text{\tiny$]$}}}}
\def\yx{{{\text{\tiny $[$}}yx{\text{\tiny$]$}}}}

\def\sgn{\operatorname{sgn}}

\def\KK{{\mathzc K\kern0pt}}
\def\vv{{\mathzc v\kern.5pt}}

\def\aq{/\kern-2pt/}

\def\La{{\Lambda}}
\def\End{{\text{\rm End}}}

\def\span{{\text{\rm span}}}
\def\wt{{\text{\rm wt}}}

\def\up{{\boldsymbol{\upsilon}}}

\def\sZ{{\mathcal Z}}
\def\sD{{\mathcal D}}

\def\la{{\lambda}}
\def\al{{\alpha}}

\def\ttM{{\mathtt M}}

\def\sfF{{\mathsf F}}

\def\fkf{{\mathfrak{f}}}

\def\sfs{{\mathsf{s}}}

\def\fkm{{\mathfrak m}}

\def\fS{{\mathfrak S}}

\def\sM{{\mathcal M}}
\def\xy{{{\text{\tiny $[$}}xy{\text{\tiny$]$}}}}
\def\yx{{{\text{\tiny $[$}}yx{\text{\tiny$]$}}}}
\def\ro{{\rm ro}}
\def\co{{\rm co}}

\def\pp{{+\!+}}

\begin{document}

\baselineskip16pt
\def\hei{\relax}

\title[Polynomial super representations]{Polynomial super representations of $U_{q}^{\text{\rm res}}(\mathfrak{gl}_{m|n})$ at roots of unity}
\author{Jie Du, Yanan Lin  and Zhongguo Zhou}
\address{J. D., School of Mathematics and Statistics,
University of New South Wales, Sydney NSW 2052, Australia}
\email{j.du@unsw.edu.au}
\address{Y. L., School of Mathematical Sciences, Xiamen University, Xiamen 361005, China}
\email{ynlin@xmu.edu.cn}
\address{Z. Z., College of Science, Hohai University, Nanjing, China}
\email{zhgzhou@hhu.edu.cn}




\begin{abstract}
As a homomorphic image of the hyperalgebra $U_{q,R}(m|n)$ associated with the quantum linear supergroup $U_\up(\mathfrak{gl}_{m|n})$, we first give a presentation for the $q$-Schur superalgebra $S_{q,R}(m|n,r)$ over a commutative ring $R$. We then develop a criterion for polynomial supermodules of $U_{q,F}(m|n)$ over a filed $F$ and use this to determine a classification of polynomial irreducible supermodules at roots of unity. This also gives classifications of irreducible $S_{q,F}(m|n,r)$-supermodules for all $r$. As an application when $m=n\geq r$ and motivated by the beautiful work \cite{bru} in the classical (non-quantum) case, we provide a new proof for the Mullineux conjecture related to the irreducible modules over the Hecke algebra $H_{q^2,F}(\fS_r)$; see \cite{Br} for a proof without using the super theory.
\end{abstract}

\thanks{
The work was supported by a 2017 UNSW Science Goldstar Grant, the Natural Science Foundation of 
China ($\#$11471269), and Jiangsu Provincial Department of Education. The authors would like to thank UNSW and Xiamen University for their hospitality 
during the writing of the paper.}

 \maketitle


%
%
%
\section{Introduction}
The Mullineux conjecture \cite{M} refers to a combinatorial algorithmic map $\la\mapsto\ttM(\la)$ on $p$-regular partitions such that if $D^\la$ is an irreducible $p$-modular representation of the symmetric group $\fS_r$ then $D^{\ttM(\la)}\cong D^\la\otimes\text{sgn}$, where sgn is the sign representation.  Building on his work on modular branching rules, Kleshchev \cite{K} developed an alternative algorithm to describe the partition associated with $D^\la\otimes\text{sgn}$. With some technical combinatorics, Ford and Kleshchev \cite{FK} then proved that Kleshchev's algorithm is equivalent to the Mullineux map, and thereby, proved the Mullineux conjecture.  See \cite{BO} for a shorter proof for the equivalence. The Hecke algebra version of this conjecture was proved by Brundan \cite{Br}. Like the $p$-modular case, quantum branching rules play a decisive role in the proof. 

In 2003, Brundan and Kujawa \cite{bru} discovered an excellent new proof for the original conjecture without using branching rules. Instead, they used representations of the general linear Lie supergroup. This proof involves a different algorithm introduced by Xu \cite{xu} for the Mullineux map and the Serganova algorithm for computing the highest weights of $w$-twisted irreducible supermodules.
The latter relies on the highest weight theory developed in \cite[\S4]{bru} associated with a representative $w$ of  an $\fS_m\!\times\fS_n$-coset. However, this theory does not seem to have a quantum analogue. Thus, generalising the work in \cite{bru} to the quantum case requires some new ideas. 

In this paper, we will use the polynomial super representation theory of the (super) quantum hyperalgebra associated with the linear Lie superalgebra $\mathfrak{gl}_{m|n}$ to give a new proof of the quantum Mullineux conjecture. Here are the main ideas to tackle the two algorithms used in \cite{bru}. First, we directly link the map $j_l$ used in Xu's algorithm to a non-vanishing condition of certain products of Gaussian polynomials which naturally occur in root vector actions on a maximal vector; see Lemmas \ref{mull} and \ref{comp}. This results in a classification of polynomial irreducible supermodules. Second, we realise the Serganova algorithm through a sequence of root vector actions on a highest weight vector. It is worth noting that the graph automorphism $\sigma$, available only when $m=n$, and a pair of Schur functors play crucial roles in the final stage of the proof.

We organise the paper as follows. We first discuss in \S2 the Lusztig $\mathbb Z[\up,\up^{-1}]$-form $U_{\up,\sZ}({m|n})$ of the quantum supergroup
$U_{\up}(\mathfrak{gl}_{m|n})$ over $\mathbb Q(\up)$ and their base change $U_{q,R}({m|n})$ to any commutative ring $R$ via $\up\mapsto q$, the quantum (super) hyperalgebras. We also display the commutation formulas of root vectors which are used throughout the paper. In \S3, we introduce the $q$-Schur superalgebra $S_{q,R}(m|n,r)$ not only as an endomorphism algebra of a module of the Hecke algebra $H_{q^2,R}$ but also as a homomorphic image of $U_{q,R}({m|n})$. By working out a presentation for $S_{q,R}(m|n,r)$ in \S4, we develop a criterion which tests when a finite dimensional weight $U_{q,F}({m|n})$-supermodule is polynomial in \S5. A classification of irreducible weight $U_{q,F}({m|n})$-supermodule is also given as an extension of its nonsuper counterpart \cite{Lu}.  

In \S6, we classify all polynomial irreducible $U_{q,F}({m|n})$-supermodules (Theorem \ref{cirr}) which are indexed by the sets used in \cite{bru}. Notably, the method here is very different from those used in \cite{bru}. As a simple application, a classification of irreducible $S_{q,F}(m|n,r)$-supermodules is given in \S7. Unlike the classification given in \cite{dgw1,dgw2}, which is independent of quantum supergroups, this classification is constructive. We further investigate the structure of $q$-Schur superalgebras through a certain filtration of ideals and  Weyl supermodules. 
The last two sections are devoted to prove the quantum Mullineux conjecture. The combinatorics of the Mullineux map,  largely following \cite{bru}, and the quantum Serganova algorithm (Proposition \ref{lowe}, Theorem \ref{sigm}) are discussed in \S8. In the last section, we introduce two Schur functors and compare their images on supermodules (Proposition \ref{difh}). The conjecture is proved in Theorem \ref{QMC}.

%
%
%

Throughout the paper, we assume that $R$ is a commutative ring with 1 of characteristic $\neq 2$. Let $q\in R$ be an invertible element. From \S5 onwards, we assume that $R=F$ is a field and $q$ is a primitive $l'$th root of unity. To include the non-roots of unity case, we set $l'=\infty$ if $q$ is not a unit of unity.

For fixed non-negative integers $m,n$ with $m+n>0$ and $i\in [1,m+n]:=\{1,2,\cdots, m+n\},$
define the parity function $i\mapsto \bar i$ by
$$
\bar i= \begin{cases}
            \bar 0, & \mbox{  if  } 1\leq i\leq m;\\
            \bar 1, & \mbox{  if  } m+1\leq i\leq m+n.
           \end{cases}
$$
For the  standard basis $\{ \epsilon_1,\cdots, \epsilon_{m+n}\}$ for $\mathbb Z^{m+n}$, 
define the ``super dot product'' by
$(\epsilon_i,\epsilon_j)=(\epsilon_i,\epsilon_j)_s=(-1)^{\bar{i}}\delta_{ij},$ and
call $ \alpha_i={ \epsilon_i}-{ \epsilon_{i+1}}, i\in [1,m+n):=[1,m+n]\backslash\{m+n\}$ simple roots.
We have positive root system
   $\Phi^+=\{\alpha_{i,j}={ \epsilon_i}-{ \epsilon_{j}}\,|\, 1\leq i<j\leq m+n \}$
and negative root system $\Phi^-=-\Phi^+.$
Define  $\bar \alpha_{i,j}=\bar i+\bar j,$ and call $\alpha_{i,j}$ an even (resp. odd) root
 if $\bar \alpha_{i,j}=\bar0$ (resp., $\bar 1$). Note that  $\alpha_{m}$
is the only odd simple root.

\vspace{.3cm}
\noindent
{\bf Acknowledgement.} The first author would like to thank Jonathan Brundan for his comments made 
at the 2016 Charlottesville and 2017 Sydney conferences. The authors also thank Weiqiang Wang and Hebing Rui for several discussions.

%
%
%
\section{The quantum hyperalgebra $U_{q,R}({m|n})$}
  Let $ \mathbb Q(\up)$ be the field of rational functions in indeterminate  $\up$ and let
$$\up_a=\up^{(-1)^{\bar a}}\quad(1\leq a\leq m+n).$$
Define the super (or graded) commutator on the homogeneous elements $X,Y$ for an (associative) superalgebra by
$$
[X,Y]=[X,Y]_s=XY-(-1)^{\bar X \bar Y}YX.
$$
     The following quantum enveloping superalgebra $U_\up(\mathfrak {gl}_{m|n})$
is defined in \cite{zrb, dew}.
\begin{definition}\label{Uglmn}
The quantum enveloping superalgebra $U_\up(\mathfrak{gl}_{m|n})$ over $ \mathbb Q(\up)$
 is generated by the homogeneous elements
\[
E_{1}, \dotsc , E_{m+n-1}, F_{1}, \dotsc , F_{m+n-1}, K_{1}^{\pm 1}, \dotsc , K_{m+n}^{\pm  1},
\] 
with a $\Z_{2}$-grading given by setting $\overline{E}_{m}=\overline{F}_{m}=\bar 1$,
 $\overline{E}_{a}= \overline{F}_{a}=\0$
 for $a \neq m$, and $\overline{K^{\pm 1}_{a}}=\0$.
 These elements are subject to the following relations:
\begin{enumerate}
\item[(QG1)]
$
K_aK_b=K_bK_a,\
K_aK_a^{-1}=K_a^{-1}K_a=1;
$
\item[(QG2)]
$
K_aE_{b}=\up^{(\varepsilon_{a}, \alpha_{b})}E_{b}K_a,\
K_aF_{b}=\up^{(\varepsilon_{a}, -\alpha_{b})}F_{b}K_a;
$
\item[(QG3)]
$
[E_{a}, F_{b}]=\delta_{a,b}\frac{K_aK_{a+1}^{-1}-K_a^{-1}K_{a+1}}{\up_a-\up_a^{-1}},
$
\item[(QG4)]
$
E_{a}E_{b}=E_{b}E_{a},
\  F_{a}F_{b}=F_{b}F_{a},$ if $|a-b|>1
;$
\item[(QG5)]
For $ a\neq m$ and $|a-b|=1,$
\begin{equation}
\begin{aligned}\notag
&E_{a}^2E_{b}-(\up_a+\up_a^{-1})E_{a}E_{b}E_{a}+E_{b}E_{a}^2=0,\\
&F_{a}^2F_{b}-(\up_a+\up_a^{-1})F_{a}F_{b}F_{a}+F_{b}F_{a}^2=0;
\end{aligned}
\end{equation}
\item[(QG6)]
$E_m^2=F_m^2=[E_{m}, E_{m-1,m+2}]=[F_{m}, E_{m+2,m-1}]=0,$
where 
\begin{equation}
\begin{aligned}\notag
E_{m-1,m+2}&=E_{m-1}E_{m}E_{m+1}-\up E_{m-1}E_{m+1}E_{m} -\up^{-1} E_{m}E_{m+1}E_{m-1}+
E_{m+1}E_{m}E_{m-1},\\
E_{m+2,m-1}&=F_{m+1}F_{m}F_{m-1}-\up^{-1} F_{m}F_{m+1}F_{m-1} -\up F_{m-1}F_{m+1}F_{m}+
F_{m-1}F_{m}F_{m+1}.\\
\end{aligned}
\end{equation}
\end{enumerate}
\end{definition}
Note that, if $a=b=m$ in (QG3), then $E_mF_m+F_mE_m=\frac{K_mK_{m+1}^{-1}-K_m^{-1}K_{m+1}}{\up-\up^{-1}}$.

%
%
%

By directly checking the relations, it is clear that there is a $\mathbb Q(\up)$-algebra automorphism (of order 4); cf. \cite[Lem. 6.5(1)]{ddp}:
\begin{equation}\label{varpi}
\varpi:U_\up(\mathfrak{gl}_{m|n})\longrightarrow U_\up(\mathfrak{gl}_{m|n}),\quad
E_a\mapsto (-1)^{\overline{a}+\overline{a+1}}F_a, F_a\mapsto E_a, K_j^{\pm1}\mapsto K_j^{\mp1}, 
\end{equation} 
and a ring anti-automophism of order 2
\begin{equation}\label{Up}
\Upsilon: U_\up(\mathfrak{gl}_{m|n})\longrightarrow U_\up(\mathfrak{gl}_{m|n}),\quad
E_a\mapsto F_a, F_a\mapsto E_a, K_j^{\pm1}\mapsto K_j^{\mp1}, \up\mapsto \up^{-1}.
\end{equation} 
When $m=n$, we have the following $\mathbb Q(\up)$-algebra automorphism induced from a ``graph automorphism''
\begin{equation}\label{sigma}
\sigma:U_{\up}(\mathfrak{gl}_{n|n})\longrightarrow U_\up(\mathfrak{gl}_{n|n}),\quad E_{a}\rightarrow F_{2n-a}, \  F_{a}\rightarrow E_{2n-a},\   K_j^{\pm1}\rightarrow
{K}_{2n+1-j}^{\mp1}.
\end{equation}
%
%
%
%
%
%

We now introduce Lusztig's $\mathcal Z$-form\footnote{This form in the nonsuper case, first introduced in \cite{Lu}, is also known as the restricted integral form in \cite[\S9.3]{cp}; compare with the non-restricted form in \cite[\S9.2]{cp}. See also their respective representation theories of their specialisations in \cite[\S\S11.1-2]{cp}.} of $U_\up(\mathfrak{gl}_{m|n})$, where $\mathcal {Z}:=\mathbb Z[\up, \up^{-1}]$. Let
$[ i ]=\frac{\up^i-\up^{-i}}{\up-\up^{-1}}$ and   $[ i ]!=[1][2]\cdots [i].$
For any integers $t\in \mathbb N, s\in \mathbb Z$,  define (symmetric) Gaussian polynomials by
\begin{eqnarray*}
{s \brack t}=\frac{[s]!}{[t]![s-t]!}=\prod_{i=1}^t
\frac{{\up^{s-i+1}-\up^{-s+i-1}}}{\up^i-\up^{-i}}.
\end{eqnarray*}
Note that, by the evaluation map from $\sZ$ to $R$ via $\up\mapsto q$,  the evaluation of the polynomial $[{s\atop t}]$ at $q$ is denoted $[{s\atop t}]_{q}$. Note also that if $q_a$ is the value of $\up_a$ at $q$ then
$$\bigg[{s\atop t}\bigg]_{q_a}=\bigg[{s\atop t}\bigg]_{q}.$$
For $c\in\mathbb Z$, $t\in\mathbb N$, set ${K_i;c\brack 0}=1$ and, for $t>0$,
\begin{eqnarray}\label{Kct}
{K_i; c \brack t}={K_i; c \brack t}_i=\prod_{s=1}^t
\frac{K_i{\up_i^{c-s+1}-K_i^{-1}\up_i^{-c+s-1}}}{\up_i^s-\up_i^{-s}}.
\end{eqnarray}
Here, we sometimes use the subscript $i$ to indicate the use of $\up_i$.
Let $E_{i}^{(M)}=\frac{E_i^M}{[M]^!}$,  $F_{i}^{(M)}=\frac{F_i^M}{[M]^!}$, and ${K_i\brack t}=
{K_i; 0 \brack t}$. 
Then $U_\up(\mathfrak{gl}_{m|n})$ has the Lusztig  $\mathcal Z$-form  $U_{\up,\mathcal Z}:=U_{\up,\mathcal Z}(m|n)$. This is the $\sZ$-subsuperalgebra generated  by
$$
\bigg\{
E_{i}^{(M)}, F_{i}^{(M)}, K_j^{\pm 1}, {K_j\brack t}
\, \bigg|  \,
t,M\in \mathbb N, 1\leq i< m+n,\, 1\leq j\leq m+n
\bigg\}.
$$
It is clear to see that the automorphisms $\varpi,\Upsilon$ 
stabilise
the $\mathcal Z$-form  $ U_{\up,\mathcal Z}$. Likewise, the graph automorphism $\sigma$ 
restricts to an $\mathcal Z$-algebra automorphism
\begin{equation}\label{intsig}
\sigma:U_{\up,\mathcal Z}({n|n})\longrightarrow U_{\up,\mathcal Z}({n|n}).
\end{equation}

We need quantum root vectors to describe PBW type bases for  $U_{\up,\sZ}({m|n})$ below.
For a root $\alpha=\epsilon_a-\epsilon_b$,  define recursively the {\it root vectors} $E_\alpha=E_{a,b}$ as follows:
\begin{equation}\notag
\aligned
E_{a,a+1}&=E_{a},\quad E_{a+1,a}=F_{a} \
\mbox{ and, for  }|a-b|>1,\\
E_{a,b}&=\left\{
\begin{array}{ll}
E_{a,c}E_{c,b}-\up_c E_{c,b}E_{a,c},& \quad \mbox{ if } a>b;\\
E_{a,c}E_{c,b}-\up_c^{-1} E_{c,b}E_{a,c},& \quad \mbox{ if } a<b.
\end{array}
\right.
\endaligned
\end{equation}
Here $c$  can be any number strictly between $a$ and $b$. Note that we have
\begin{equation}\label{fab}
\Upsilon(E_{a,b})=E_{b,a},\quad \varpi(E_{a,b})=\pm(-q_a)^{f_{a,b}}E_{b,a}\text{ for some }f_{a,b}\in\mathbb N.
\end{equation}

The following three sets of commutation formulas for divided powers of root vectors 
$E_{a,b}^{(M)}:=\frac{E_{a,b}^M}{[M]^!}$ ($a\neq b$, $M\geq1$) in $U_{\up,\mathcal Z}$, are given in \cite{tk}. They continue to hold in the specialisation to an arbitrary commutative ring $R$ via $\up\mapsto q\in R$:
\begin{equation}\label{over F}
U_{q,R}=U_{q,R}(m|n) =U_{\up,\mathcal{Z}}\otimes_\mathcal{Z} R.
\end{equation}
Following \cite[\S3]{Br}, we call $U_{q,R}$ the {\it quantum (super) hyperalgebra} associated with $U_\up(\mathfrak{gl}_{m|n})$, which is also denoted by $U_q^{\text{res}}(\mathfrak{gl}_{m|n})$ in \cite[\S9.3]{cp}.
For notational simplicity, we  write $X=X\otimes 1$ for all $X\in U_{\up,\mathcal Z}$, We also 
set  $\varpi_R=\varpi\otimes \text{id}_R$, $\Upsilon_R$ and $\sigma_R$ to denote the corresponding automorphisms. 
For example, $\sigma_R:U_{q,R}(n|n)\longrightarrow U_{q,R}(n|n)$ satisfies
\begin{equation}\label{sigR}
(E_{i}^{(M)}, F_{i}^{(M)}, K_j^{\pm 1}, {K_j\brack t})\longmapsto
(F_{2n-i}^{(M)}, E_{2n-i}^{(M)}, K_{2n+1-j}^{\mp 1}, {K_{2n+1-j}^{-1}\brack t}).
\end{equation}
Note that applying $\varpi_R,\Upsilon_R$ to the commutation relations below may produces other commutation relations in $ U_{q,R}$.

\begin{proposition}\label{KE} For any $1\leq a,b,c\leq m+n$, we have, in $U_{q,R}$,
$$K_aE_{b,c}=\begin{cases} q_a^{\varepsilon_a}E_{b,c}K_a,&\text{ if } a=b\text{ or }c;\\
E_{b,c}K_a,&\text{ if } a\neq b,c,\end{cases} \quad\text{where}\quad \varepsilon_a=\begin{cases} 1,&\text{ if } a=b;\\
-1,&\text{ if } a=c.\end{cases}$$
\end{proposition}
\begin{proposition}[{\cite[(27)\&Proposition 3.8.1]{tk}}]\label{ppco}
Let $E_{a,b}$ and $E_{c,d}$ be two root vectors with $a<b$ and $c<d$, 
and let $M,N \geq 1.$  We then have the following commutation formulas in $U_{q,R}$.
\begin{enumerate}
\item[(0)] $E_{a,b}^2=0\;\; \text{ for all odd root }\;\; \alpha=\epsilon_a-\epsilon_b;$
\item If $b<c$ or $c<a<b<d$, then
\[
E_{a,b}^{(M)}E_{c,d}^{(N)}=
\begin{cases}
(-1)^{\overline{E}_{a,b}\overline{E}_{c,d}}E_{c,d}E_{a,b}, &\text{if } M=N=1;\\
E_{c,d}^{(N)}E_{a,b}^{(M)},  &\text{otherwise}.
\end{cases}
\]
\item If $a=c<b<d$ or $a<c<b=d$, then
\[
E_{a,b}^{(M)}E_{c,d}^{(N)}=\begin{cases}
(-1)^{\overline{E}_{a,b} \overline{E}_{c,d}}q_bE_{c,d}E_{a,b}, &\text{if } M=N=1, b=d;\\
(-1)^{\overline{E}_{a,b} \overline{E}_{c,d}}q_aE_{c,d}E_{a,b}, &\text{if } M=N=1,a=c;\\
q_b^{MN} E_{c,d}^{(N)}E_{a,b}^{(M)},&\text{otherwise}.
\end{cases}
\]
\item If $a<b=c<d$, then
\[
E_{a,b}^{(M)}E_{c,d}^{(N)}=\begin{cases}
 E_{a,d}+q_c^{-1}E_{c,d}E_{a,b}, & \text{if } M=N=1;\\
\sum_{t=0}^{\min(M,N)}q_{b}^{-(N-t)(M-t)}E_{c,d}^{(N-t)}E_{a,d}^{(t)}E_{a,b}^{(M-t)},
&\text{otherwise}.\end{cases}
\]
\item If $a<c<b<d$, then
\[
E_{a,b}^{(M)}E_{c,d}^{(N)}=\begin{cases}
 (-1)^{\overline{E}_{a,b}\overline{E}_{c,d}}E_{c,d}E_{a,b}+(q_b-q_b^{-1})E_{a,d}E_{c,b}, &\text{if } M=N=1;\\
\sum_{t=0}^{\min(M,N)}q_{b}^{\frac{t(t-1)}{2}}(q_b-q_b^{-1})^t[t]_q!E_{c,b}^{(t)}E_{c,d}^{(N-t)}E_{a,b}^{(M-t)}E_{a,d}^{(t)},
&\text{otherwise}.\end{cases}
\]
\end{enumerate}
\end{proposition}

For $\alpha=\epsilon_i-\epsilon_j\in \Phi$, let $K_\alpha=K_{i,j}=K_iK_j^{-1}$ and define ${K_{i,j};c\brack t}={K_{i,j};c\brack t}_i$ as in \eqref{Kct}, replacing $K_i$ there by $K_{i,j}$. 

\begin{proposition}[{ \cite[(29)\&Proposition 3.9.1]{tk}}]\label{cqrv}
Let $E_{a,b}$ and $E_{d,c}$ be two root vectors  with $a<b$ and $c<d$,  and let $M,N \geq 1$. 
We then have the following commutation formulas in $U_{q,R}$.
\begin{enumerate}
\item  If $b\leq c$ or $c<a<b<d$, then
\[
E_{a,b}^{(M)}E_{d,c}^{(N)}=\begin{cases}(-1)^{\bar E_{a,b}\bar E_{d,c}}E_{d,c}E_{a,b},&\text{if }M=N=1;\\
E_{d,c}^{(N)}E_{a,b}^{(M)},&\text{otherwise}.\end{cases}
\]
\item If $a<c<b=d$
\[
E_{a,b}^{(M)}E_{d,c}^{(N)}=\begin{cases}(-1)^{\bar E_{a,b}\bar E_{d,c}}E_{d,c}E_{a,b}+K_{c,d}E_{a,c},&\text{if }M=N=1;\\
\displaystyle \sum_{t=0}^{\min(M,N)} 
     q_{b}^{-t(N-t)}E_{d,c}^{(N-t)}K_{c,d}^{t}E_{a,b}^{(M-t)}E_{a,c}^{(t)},&\text{otherwise}.\end{cases}
\]
\item If $a=c<b<d$, then
\[
E_{a,b}^{(M)}E_{d,c}^{(N)}=\begin{cases}(-1)^{\bar E_{a,b}\bar E_{d,c}}E_{d,c}E_{a,b}-(-1)^{\bar E_{a,b}\bar E_{d,c}}K_{a,b}E_{d,b},&\text{if }M=N=1;\\
\displaystyle 
 \sum_{t=0}^{\min(M,N)} (-1)^tq_{b}^{-t(M-1-t)}
    E_{d,b}^{(t)}E_{d,c}^{(N-t)}K_{a,b}^{t}E_{a,b}^{(M-t)},&\text{otherwise}.\end{cases}
\]
\item  If $a<b$, then
\[
E_{a,b}^{(M)}E_{b,a}^{(N)}=\begin{cases}(-1)^{\bar E_{a,b}\bar E_{b,a}}E_{b,a}E_{a,b}+(q_a-q_a^{-1})^{-1}(K_{a,b}-K_{a,b}^{-1}),&\text{if }M=N=1;\\
\displaystyle
 \sum_{t=0}^{\min(M,N)} 
       E_{b,a}^{(N-t)}\begin{bmatrix} K_{a,b}; 2t-M-N \\ t \end{bmatrix} E_{a,b}^{(M-t)},&\text{otherwise}.\end{cases}
\]
\item If $a<c<b<d$, then
\[
E_{a,b}^{(M)}E_{d,c}^{(N)}=\begin{cases}(-1)^{\bar E_{a,b}\bar E_{d,c}}E_{d,c}E_{a,b}-(q_b-q_b^{-1})^{-1}K_{c,b}E_{a,c}E_{d,b},\text{ if }M=N=1;\\
\displaystyle \sum_{t=0}^{\min(M,N)} 
    (-1)^tq_{b}^{\frac{-t(2N-3t-1)}{2}}(q_b-q_b^{-1})^{t}[t]_q!E_{d,c}^{(N-t)}
       E_{d,b}^{(t)}K_{c,b}^{t}E_{a,b}^{(M-t)}E_{a,c}^{(t)},\text{ o.w}.\end{cases}
\]

\end{enumerate}
\end{proposition}

The commutation formulas can easily be used to obtained  the so-called PBW bases.
Let
\begin{equation}\label{Pmn}
P(m|n)=\{
A=(A_\alpha)_{\alpha\in \Phi}
\,|\,
A_\alpha\in \mathbb N\mbox{ if } \bar{\alpha}=\bar0
\mbox{ and }  A_\alpha\in\{0,1\} \mbox{ if } \bar{\alpha}=\bar1
\}.
\end{equation}
For $A\in P(m|n)$ and any fixed ordering on $\Phi^+$ and $\Phi^-$,  let
\begin{equation}\label{E_A}
E_A=\prod_{\alpha\in\Phi^+}E_{\alpha}^{(A_\alpha)},
\
F_A=\prod_{\beta\in\Phi^-}E_\beta^{(A_\beta)}.
\end{equation}
Then  $U_{q,R}$ has an (integral) $R$-basis (see \cite[\S3.10]{tk})
\begin{equation}\label{ubas}
 \left\{
 E_A\prod_{a=1}^{m+n}\left(K_a^{\sigma_a}  {K_a\brack \mu_a}\right) F_A
 \,\bigg|\,
 A\in P(m|n),\sigma_a\in\{0,1\},\mu\in \mathbb N^{m+n}
\right \}.
 \end{equation}

%
%
%

We define analogously positive part, negative part and zero part
as in the non-super case:
$U_{q,R}^+, { U}_{q,R}^-,{ U}_{q,R}^0.$ Denote
$U_{q,R}^{\geq0}= { U}_{q,R}^+{ U}_{q,R}^0.$

\begin{remark}\label{Lu90} In \cite[\S2.3, Thm~4.5]{Lu90}, Lusztig gave a presentation for the $\sZ$-form $U_\sZ$ of a quantum group associated with a symmetric Cartan matrix. It should not be hard to generalise this work to get a presentation for $U_{\up,\sZ}(m|n)$ and for $U_{\up,R}(m|n)$. 
\end{remark}

%
%
%

\section{The $q$-Schur superalgebras $S_{q,R}(m|n,r)$}
We first review the definition of $q$-Schur superalgebras in terms of an endomorphism algebra of a $q$-permutation module over the Hecke algebra $H_{q^2,R}$ associated with the symmetric group $\fS_r$ on $r$ letters. Let $S=\{s_i=(i,i+1)\}$ be the generating set of basic transpositions.

The Hecke algebra ${H}_{q^2,R}={H}_{q^2,R}(r)$ is the $R$-algebra with generators $T_i$, $1\leq i\leq r-1$, which subject to the relation
$$T_iT_j=T_jT_i, |i-j|>1;\quad T_iT_jT_i=T_jT_iT_j, |i-j|=1;\quad T_i^2=(q^2-1)T_i+q^2.$$
By setting $T_{s_i}=T_i$ and $T_w=T_{i_1}\cdots T_{i_l}$ if $w=s_{i_1}\cdots s_{i_l}$ is a reduced expression, ${H}_{q^2,R}$ is a free $R$-module with basis $\{T_w\mid w\in \fS_r\}$ and the
multiplication satisfies the rules: for $s\in S$,
\begin{equation}
T_wT_s=\left\{\begin{aligned} &T_{ws},
&\mbox{if } \ell(ws)>\ell(w);\\
&(q^2-1)T_w+q^2 T_{ws}, &\mbox{if } \ell(ws)<\ell(w).
\end{aligned}
\right.
\end{equation}
Since $q^2$ is invertible, it follows that $T_i^{-1}$ exists and every basis element $T_w$ is invertible.



The Hecke algebra $H_{q^2,R}$ admits the following $R$-algebra automorphism
\begin{equation}\label{sharpH}
(-)^\sharp:H_{q^2,R}\longrightarrow H_{q^2,R}, \quad T_i\longmapsto (q^{2}-1)-T_i
\end{equation}
Since the symmetric group $\fS_r$ is the Coxeter group associated with Coxeter graph
$$\underset 1\circ\!\!-\!\!\!-\!\!\!-\!\!\underset2\circ\!\!-\!\!\!-\!\!\!-\!\!\underset3\circ\!\!-\!\!\!-\!\!\cdots\cdots \!\!-\!\!\!-\!\!\!\!\underset{r-2}\circ\!\!\!\!-\!\!\!-\!\!\!-\!\!\!\underset{r-1}\circ,$$
the graph automorphism $(-)^\dagger$ sending $i$ to $r-i$ induces a group automorphism and an $R$-algebra automorphism
\begin{equation}\label{dagger}
\aligned
(-)^\dagger&:\fS_r\longrightarrow \fS_r,\quad s_i\longmapsto s_{r-i};\\
(-)^\dagger&:H_{q^2,R}\longrightarrow H_{q^2,R},\quad T_i\longmapsto T_{r-i}.
\endaligned
\end{equation}

For a composition $\la$ of $r$, i.e., $\la$ is an element of the set
$$\Lambda(N,r)=\{(\lambda_1,\dots,\lambda_{N})\in\mathbb N^N\mid\sum_{i=1}^N\lambda_i=r\},\text{ for some }N,$$
let $\fS_\la$ be the associated parabolic (or standard Young) subgroup and let $\sD_\la:=\mathcal{D}_{\fS_\la}$ be
the set of all shortest coset representatives of the right
cosets of $\fS_\la$ in $\fS_r$. Let
$\mathcal{D}_{\lambda\mu}=\mathcal{D}_\lambda\cap\mathcal{D}^{-1}_{\mu}$ be the set of the shortest
$\fS_\lambda$-$\fS_\mu$ double coset representatives.

For $\la,\mu\in\La(N,r)$ and
$d\in\mathcal{D}_{\la\mu}$, the subgroup 
$$\fS_{\la d\cap\mu}:=\fS_\la^d\cap
\fS_\mu=d^{-1}\fS_\la d\cap \fS_\mu$$ is a parabolic subgroup associated
with the composition $\la d\cap\mu$ which can be easily described in terms of the following 
$N\times N$-matrix $A=(a_{i,j})$, where $a_{i,j}=|R^\la_i\cap d(R^\mu_j)|$: 
if $\nu^{(j)}=(a_{1,j},a_{2,j},\ldots,a_{N,j})$ denotes the $j$th column of $A$, then
\begin{equation}\label{ladmu}
 \la d\cap\mu=(\nu^{(1)},\nu^{(2)},\ldots,\nu^{(N)}).
  \end{equation}
Putting $\jmath(\la,d,\mu)=\big(|R^\la_i\cap d(R^\mu_j)|\big)_{i,j}$, we obtain a bijection
\begin{equation}\label{jmath}
\jmath:\bigcup_{ \la,\mu\in\La(N,r)}\{(\la,d,\mu)\mid d\in\sD_{\la\mu}\}\longrightarrow \sM(N,r),
\end{equation}
where $\sM(N,r)$ is the subset of the $N\times N$ matrix ring $M_N(\mathbb N)$ over $\mathbb N$ consisting of matrices  $A=(a_{i,j})$
whose entries sum to $r$, i.e., $|A|:=\sum_{i,j}a_{i,j} =r$. Note that, if $\jmath(\la,d,\mu)=A$, then 
\begin{equation}\label{ro co}
\la=\ro(A):=(\sum_{j=1}^Na_{1,j},\ldots,\sum_{j=1}^Na_{N,j}),\,\quad\,\mu=\co(A):=(\sum_{i=1}^Na_{i,1},\ldots,\sum_{i=1}^Na_{i,N}).
\end{equation}

For $A=(a_{i,j})\in\sM(N,r)$, let $A^\dagger=(a_{i,j}^\dagger)$, 
where $a_{i,j}^\dagger=a_{N-j+1,N-i+1}$. So $A^\dagger$ is obtained by two transposes along diagonal 
and anti-diagonal respectively.
We thus have a bijection
  $$(-)^\dagger:\sM(N,r)\longrightarrow \sM(N,r),\quad A\longmapsto A^\dagger,$$
and $\jmath(\la^\dagger,d^\dagger,\mu^\dagger)=A^\dagger$, where $\nu^\dagger$ denotes the composition 
obtained by reversing the sequences $\nu$, i.e., 
$$\nu^\dagger=(\nu_N,\ldots,\nu_2,\nu_1),\text{ if }\nu=(\nu_1,\nu_2,\ldots,\nu_N).$$ 

For the description of a super structure, we consider two nonnegative integers $m,n$. Thus, a composition $\la$ of $m+n$ parts will be written as
$$\la=(\lambda^{(0)}|\lambda^{(1)})=(\lambda^{(0)}_1,\lambda^{(0)}_2,\cdots,\lambda^{(0)}_m|\lambda^{(1)}_1,
\lambda^{(1)}_2,\cdots,\lambda^{(1)}_n)$$ to indicate the``even'' and ``odd'' parts of $\la$. Let
\begin{equation*}
\aligned
\Lambda(m|n,r):&=\La(m+n,r)
=\bigcup_{r_1+r_2=r}(\La(m,r_1)\times\La(n,r_2)),\\
\Lambda(m|n):&=\bigcup_{r\geq0}\La(m|n,r)=\mathbb N^{m+n}.
\endaligned
\end{equation*}

For $\lambda=(\lambda^{(0)}\mid\lambda^{(1)})\in\Lambda(m|n,r)$, we also write
\begin{equation}\label{notation}
\fS_\la=
\fS_{\la^{(0)}}\fS_{\la^{(1)}}\cong \fS_{\la^{(0)}}\times \fS_{\la^{(1)}},
\end{equation}
 where
$\fS_{\lambda^{(0)}}\leq{\mathfrak S}_{\{1,2,\ldots,|\lambda^{(0)}|\}}$
and
$\fS_{\lambda^{(1)}}\leq{\mathfrak S}_{\{|\lambda^{(0)}|+1,\ldots,r\}}$ are the even and odd parts of $\fS_\la$, respectively. Define
\begin{equation}
\xy_\la:=x_{\la^{(0)}}y_{\la^{(1)}},\quad \yx_\la=y_{\la^{(0)}}x_{\la^{(1)}},
\end{equation}
 where $x_{\lambda^{(i)}}=\sum_{w\in \fS_{\lambda^{(i)}}}T_w, \;\;y_{\lambda^{(i)}}=\sum_{w\in
\fS_{\lambda^{(i)}}}(-q^2)^{-\ell(w)}T_w.$


The endomorphism algebra
\begin{equation}\label{Tmnr}
S_{q,R}=S_{q,R}(m|n,r):=\End_{H_{q^2,R}(r)}\bigg(\bigoplus_{\lambda\in\Lambda(m|
n,r)}\xy_\la {H}_{q^2,R}(r)\bigg)
\end{equation}
is called the {\it$q$-Schur superalgebra} of degree ($m|n,r$).

By definition, for $\la=(\lambda^{(0)},\lambda^{(1)})$, $\la^\dagger=(\lambda^{(1)\dagger},\lambda^{(0)\dagger})$. Let $\la^+=(\lambda^{(0)\dagger},\lambda^{(1)\dagger})$. Then
$$(\xy_\la)^\dagger=(x_{\la^{(0)}})^\dagger(y_{\la^{(1)}})^\dagger=y_{\la^{(1)\dagger}}x_{\la^{(0)\dagger}}=\yx_{\la^\dagger}.$$
Since $\fS_{\la^\dagger}$ and $\fS_{\la^+}$ are conjugate, there exists $d\in\fS_r$ such that
$\yx_{\la^\dagger}T_d=T_d \xy_{\la^+}$. Hence, $\yx_{\la^\dagger}H_{q^2,R}=T_d\xy_{\la^+}H_{q^2,R}\cong \xy_{\la^+}H_{q^2,R}$. Now, we see the following easily.

\begin{lemma}\label{xy to yx} For $m=n$, we may identify $S_{q,R}(n|n,r)$ with the endomorphism algebra 
$\End_{H_{q^2,R}}\bigg(\bigoplus_{\lambda\in\Lambda(n|n,r)}\yx_\la {H}_{q^2,R}(r)\bigg)$. In particular, the isomorphism $(\ )^\dagger$ in \eqref{dagger} induces an isomorphism of right $H_{q^2,R}$-modules 
$$f:\bigoplus_{\lambda\in\Lambda(n|
m,r)}\xy_{\la} {H}_{q^2,R}\longrightarrow\bigoplus_{\lambda\in\Lambda(m|
n,r)}\yx_{\la^\dagger}{H}_{q^2,R},\quad m\longmapsto m^\dagger,$$
which further results in an $R$-algebra automorphism
\begin{equation}\label{barsig}
(\ )^\dagger:S_{q,R}(n|n,r)\longmapsto S_{q,R}(n|n,r), \phi\longmapsto f\phi f^{-1}.
\end{equation}
\end{lemma}

Following \cite[(5.3.2)]{dr}, for $\lambda,\mu\in\Lambda(m|n,r)$,
 define
 $$\mathcal{D}^\circ_{\lambda\mu}=\{d\in\sD_{\la\mu}\mid\fS_{\la^{(i)}}^d\cap\fS_{\mu^{(j)}}=1\;\forall \bar i+\bar j=1\}.$$
 Then all $\jmath(\lambda,d,\mu)$ with $\lambda,\mu\in\Lambda(m|n,r)$, $d\in \mathcal{D}^\circ_{\lambda\mu}$ form the matrix set
 \begin{equation}\label{M(m|n,r)}
 \aligned
 \sM(m|n,r)&=\{A=(a_{ij})\in M_{m+n}(\mathbb N)\mid a_{i,j}\in\{0,1\}\;\forall \bar i+\bar j=1,|A|=r\},\\
 \sM(m|n)&=\bigcup_{r\geq0} \sM(m|n,r).\endaligned
 \end{equation}
We may interpret an element $(A_\al)_{\al\in\Phi}\in P(m|n)$ in \eqref{Pmn} as a matrix $A=(A_{i,j}) \in \sM(m|n)$, where $A_{i,j}=A_\al$ if $\al=\epsilon_i-\epsilon_j$ and $A_{i,i}=0$ for all $i$.

For $A=\jmath(\lambda,d,\mu)$, putting
$$T_{\fS_\lambda d \fS_\mu}:=\sum_{\substack{w_0\in \fS_{\mu^{(0)}},w_1\in \fS_{\mu^{(1)}}\\ w_0w_1
\in \fS_\mu\cap \mathcal{D}_{\la d\cap\mu}}}(-\up^2)^{-\ell(w_1)}x_{\lambda^{(0)}}y_{\lambda^{(1)}}T_dT_{w_0}T_{w_1},$$
there exists an $\mathcal{H}$-homomorphism $\phi_A:=\phi^d_{\lambda\mu}$ defined by
$$\phi_{\la\mu}^d(x_{\alpha^{(0)}}y_{\alpha^{(1)}}h)=\delta_{\mu,\alpha}T_{\fS_\lambda d
\fS_\mu}h, \forall \alpha\in\Lambda(m|
n,r),h\in\mathcal{H}.$$

Let $d^{(0)}$ (resp. $d^{(1)}$) to be the longest element in the double coset 
$\fS_{\lambda^{(0)}} d \fS_{\mu^{(0)}}$(resp. $\fS_{\lambda^{(1)}} d \fS_{\mu^{(1)}}$ ). 
Following \cite[(6.0.2)]{dr}, let $\mathcal{T}_A=\up^{-l(d^{(0)})+l(d^{(1)})-l(d)}T_{\fS_\lambda d \fS_\mu}.$ Then
$$
[A]=\up^{-l(d^{(0)})+l(d^{(1)})-l(d)+l(w_{0,\mu^{(0)}})-l(w_{0,\mu^{(1)}})}\phi_A,
$$
where $w_{0,\la}$ denotes the longest element in $\fS_\la$, is the map $\mathcal T_{\fS_\mu}\mapsto \mathcal{T}_A$. The first assertion of the following result is
 given in {\cite[5.8]{dr}}. 
\begin{lemma}\label{DR5.8} The set $\{\phi_A\mid
A\in \sM(m|n,r)\}$
forms a $R$-basis for $S_{q,R}(m|n,r )$. Moreover, we have 
$[A]^\dagger=[A^\dagger]$.
\end{lemma}
\begin{proof}It suffices to prove the last statement for $R=\sZ$. Let $A=\jmath(\lambda,d,\mu)$. We have
$$
\aligned[]
[A]^\dagger(\mathcal T_{\fS_{\mu^\dagger}})&=\up^{-l(d^{(0)})+l(d^{(1)})-l(d)}(\phi_A(\xy_\mu))^\dagger
                   =\up^{-l(d^{(0)})+l(d^{(1)})-l(d)}(T_{\fS_\lambda d \fS_\mu})^\dagger\\
&=\frac{\up^{-l(d^{(0)})+l(d^{(1)})-l(d)}}{P_{\nu}(\up^2)}(\xy_\la T_d \xy_\mu)^\dagger\quad(\nu=\la d\cap\mu)\\
&=\frac{\up^{-l(d^{(0)})+l(d^{(1)})-l(d)}}{P_{\nu^\dagger}(\up^2)}(\yx_{\la^\dagger} 
                      T_{d^\dagger} \yx_{\mu^\dagger})\\
&=\up^{-l(d^{(0)})+l(d^{(1)})-l(d)}T_{\fS_{\la^\dagger}d^\dagger\fS_{\mu^\dagger}}
       =\mathcal T_{\fS_{\la^\dagger}d^\dagger\fS_{\mu^\dagger}},
\endaligned
$$
where $P_\nu(\up^2)=\sum_{w_0\in\fS_{\nu^{(0)}},w_1\in\fS_{\nu^{(1)}}}\up^{2l(w_0)}(\up^{-1})^{2l(w_1)}=P_{\nu^\dagger}(\up^2)$ and the last equality is seen from the fact that $\ell({}d^{(i)})=\ell(d^{(i)\dagger})$ for $i=0,1$.
\end{proof}

El Turkey and Kujawa  (\cite[Thm 3.3.1]{tk})  gave a presentation of the $\up$-Schur superalgebra  
$S_\up(m|n, r)$ over $\mathbb Q(\up)$.
They proved that $S_\up(m|n, r)$  is generated by the similar generators and defining
relations for ${U}_\up(m|n)$  over $ \mathbb Q(\up)$  along with relations:
\begin{equation}\label{ETK}
 K_1\cdots K_mK_{m+1}^{-1}\cdots K_{m+n}^{-1}-\up^{r}=0,
\quad (K_a-1)(K_a-\up_a)\cdots(K_a-\up_a^r)=0.
\end{equation}
Thus, if $I_r$ denotes the ideal of ${U}_\up(m|n)$ generated by   $K_1\cdots K_mK_{m+1}^{-1}\cdots K_{m+n}^{-1}-\up^{r}$ 
and $ (K_a-1)(K_a-\up_a)\cdots(K_a-\up_a^r)$, $1\leq a\leq m+n$, then
${U}_\up(m|n)/I_r\cong   S_\up(m|n, r) .$
So we have an algebra epimorphism (see \cite[(20)]{tk} or \cite[Cor. 6.4]{dg}):
\begin{equation}\label{eta_r}
\eta_r: U_\up(m|n)\longrightarrow S_\up(m|n,r).
\end{equation}
In particular, $S_\up(m|n, r)$  has generators:
$$
\tte_{a}=\eta_r(E_{a}),\
\ttf_{a}=\eta_r(F_{a}),\
\ \ttk_j^{\pm 1}=\eta_r(K_j^{\pm 1}).
$$


%
%
%

Put $\tte_{a}^{(M)}=\eta_r(E_{a}^{(M)}),{\ttk_j\brack t} =\eta_r({K_j\brack t})$, etc., and 
let ${S}_{\up,\mathcal Z}={S}_{\up,\mathcal Z}(m|n,r)$ be the $\sZ$-subalgebra of $S_\up(m|n, r)$ generated  by 
\begin{equation}\label{Sgenerators}
\bigg\{\tte_{a}^{(M)}, \ttf_{a}^{(M)}, \ttk_j^{\pm 1}, {\ttk_j\brack t}\;\bigg|\;
t,M\in \mathbb N, 1\leq a< m+n,\, 1\leq j\leq m+n\bigg\}.
\end{equation}
Then ${S}_{\up,\sZ}$ has a $\sZ$-basis of (see \cite[Thm 3.12.1]{tk})
\begin{equation}\label{ibas}
{\bf  Y}=\bigcup\left\{
 \tte_A 1_{\lambda} \ttf_A
 \,|\,
 A\in P(m|n), \lambda\in \Lambda(m|n,r),  \chi(E_AF_A)\preceq\lambda
\right \},
\end{equation}
where  $\chi$ is the content function defined in \cite[3.11]{tk}\footnote{If we identify $A$ with a matrix $(m_{i,j})$, then the $h$th component $\chi (E_AF_A)_h=\sum_{i< h}(m_{i,h}+m_{h,i})$.} and $ \tte_A, \ttf_A$ are images of
the elements $E_A,F_A$ defined in \eqref{E_A}. (Here $\mu\preceq\la$ mean $\mu_i\leq\la_i$ for all $i$.) 


For any commutative ring $R$ and any invertible element $q\in R$, base change via the specialisation 
$\sZ\to R, \up\mapsto q$ results in $R$-algebra 
$${S}_{q,R}=S_{q,R}(m|n,r) \cong{S}_{\up,\sZ}(m|n,r)\otimes_\sZ R.$$ 
Moreover, by restriction and specialisation, the map $\eta_r$ in \eqref{eta_r} induces an $R$-algebra epimorphism (see \cite[Cor. 8.4]{dg}):
\begin{equation}\label{etaR}
\eta_{r, R}:=\eta_{r}\otimes 1: {U}_{q,R}(m|n) \longrightarrow {S}_{q,R}(m|n,r).
\end{equation}
Like in \S2, we will also abuse $X$ as $X\otimes 1$ for simplicity. Thus, ${S}_{q,R}(m|n,r)$ is generated by the elements in \eqref{Sgenerators}.



%
%
%
\section{Presenting ${ S}_{q,R}(m|n,r)$ over a commutative ring $R$}
For any $\mu\in \Lambda(m|n),$ let
${K\brack \mu}=\prod_{a=1}^{m+n}{K_a\brack \mu_a}.$
Let   $J_r=J_{r,R}$ be the ideal of $U_{q,R}=U_{q,R}(m|n)$ generated by
\begin{equation}\label{J_r}
\aligned
1-\sum_{\lambda\in\Lambda(m|n,r)} {K\brack\lambda},\;\;
K_a^{\pm 1} {K\brack\lambda}-q_a^{\pm \lambda_a}{K\brack\lambda},\;\;
{K_a;c\brack t}  {K\brack\lambda} -{\lambda_a+c\brack t}_{q}  {K\brack\lambda},
 \endaligned
 \end{equation}
 where  $1\leq a\leq m+n, t\in\mathbb N, c\in \mathbb Z,\lambda\in\Lambda(m|n,r)$.
 Let  
 $\pi_{r, R}:  { U}_{q,R} \to {\overline U}_{q,R}:={ U}_{q,R}/J_r
$
be the natural homomorphism and put
$$
\sfe^{(M)}_{a,b}=\pi_{r, R}(E^{(M)}_{a,b}),\
\sfk_a^{\pm 1}=\pi_{r, R}(K_a^{\pm 1}),\
{\sfk_a\brack t}=\pi_{r, R}({K_a\brack t}),\
{\sfk\brack\lambda}=\pi_{r, R}({K\brack\lambda}).
$$
\begin{lemma}\label{reab} For any $\lambda\in\Lambda(m|n,r)$, let $1_\lambda:={\sfk\brack\lambda}$. Then
the following hold in ${\overline U}_{q,R}$:
\begin{enumerate}
\item $\sum_{\lambda\in\Lambda(m|n,r)} {1_\lambda}=1$;
\item $\sfk_a^{\pm 1} {1_\lambda}=q_a^{\pm \lambda_a}{1_\lambda},$
${\sfk_a;c\brack t}  {1_\lambda}={\lambda_a+c\brack t}_{q}  {1_\lambda},$ for all $1\leq a\leq m+n, \ t\in\mathbb N,
 c\in \mathbb Z.$
\item $\sfk_1\cdots \sfk_m\sfk_{m+1}^{-1}\cdots \sfk_{m+n}^{-1}=q^{r}.$
\item
$(\sfk_a-1)(\sfk_a-q_a)\cdots(\sfk_a-q_a^r)=0, $ \ $
       1\leq a\leq m+n.$
\item
${\sfk\brack\mu} {1_\lambda}={\lambda\brack \mu}_q {1_\lambda}$ for all $\mu\in\Lambda(m|n)$,
where ${\lambda\brack \mu}_q=\prod_{a=1}^{m+n}{\lambda_a\brack \mu_a}_q.$
Hence, $1_\lambda1_{\mu}=\delta_{\mu,\lambda}{1_\lambda}$. Moreover,
${\sfk\brack\mu}=\sum_{\lambda\in\Lambda(m|n,r)}{\lambda\brack \mu}_q {\sfk\brack\lambda}$ and
${\sfk\brack\mu}=0, \mbox{ if   } |\mu|>r$.
\item For each $\alpha\in\Phi$,
$
\sfe_{\alpha}^{(M)}{1_\la} = \left\{ \begin{array}{lll}
             1_{\la+M\alpha}\sfe_{\alpha}^{(M)}, & \mbox{  if  }  \la+M\alpha\in\Lambda(m|n,r),\\
            0,                              &   \mbox{ otherwise,}
           \end{array}
         \right.
$ and $1_\la E_\alpha^{(M)}=0$ if $\la-M\alpha\not\in\Lambda(m|n,r)$.

\end{enumerate}
\end{lemma}

%
%

\begin{proof}The relations (1) and (2) are clear from the definition, while (3) and (4) follow from (1) and (2) as
\begin{equation}\notag
\begin{aligned}
&\sfk_1\cdots \sfk_m\sfk_{m+1}^{-1}\cdots \sfk_{m+n}^{-1}
=\sum_{\lambda\in\Lambda(m|n,r)}\sfk_1\cdots \sfk_m\sfk_{m+1}^{-1}\cdots \sfk_{m+n}^{-1} {\sfk\brack\lambda}\\
&=\sum_{\lambda\in\Lambda(m|n,r)}q_1^{\lambda_1}\cdots 
          q_m^{\lambda_m}q_{m+1}^{-\lambda_{m+1}}\cdots q_{m+n}^{-\lambda_{m+n}} {\sfk\brack\lambda}
=q^{r}\sum_{\lambda\in\Lambda(m|n,r)}
          {\sfk\brack\lambda}=q^{r},
\end{aligned}
\end{equation}
\begin{equation}\notag
\begin{aligned}
&(\sfk_a-1)(\sfk_a-q_a)\cdots(\sfk_a-q_a^r)
=\sum_{\lambda\in\Lambda(m|n,r)}(\sfk_a-1)(\sfk_a-q_a)\cdots(\sfk_a-q_a^r) {\sfk\brack\lambda}\\
&=\sum_{\lambda\in\Lambda(m|n,r)}(q_a^{\lambda_{a}}-1)(q_a^{\lambda_{a}}-q_a)\cdots
(q_a^{\lambda_{a}}-q_a^r) {\sfk\brack\lambda}
=0.\\
\end{aligned}
\end{equation}
Similarly, (5) is seen as follows:
$${\sfk\brack\mu} {\sfk\brack\lambda}=\prod_{a=1}^{m+n}{\sfk_a\brack \mu_a}{\sfk\brack\lambda}
=\prod_{a=1}^{m+n}{\lambda_a\brack \mu_a}_q {\sfk\brack\lambda}
={\lambda\brack \mu}_q {\sfk\brack\lambda},$$
and ${\sfk\brack\mu}=
{\sfk\brack\mu}(\sum_{\lambda\in\Lambda(m|n,r)} {\sfk\brack\lambda})=
\sum_{\lambda\in\Lambda(m|n,r)}{\lambda\brack \mu}_q {\sfk\brack\lambda},$
which is 0 if $|\mu|>r$, as in this case there is $i$ such that $\lambda_i<\mu_i$ and so ${\lambda_i\brack \mu_i}_q=0.$ Consequently, ${1_\mu} {1_\lambda}=\delta_{\mu,\lambda}{1_\lambda}$ if $|\mu|=|\lambda|=r$.

It remains to prove (6). Recall the relations in $U_{\up,\sZ}$:
$
K_bE_{b,c}=\up_bE_{b,c}K_b, K_cE_{b,c}=\up_c^{-1}E_{b,c}K_c
$
and
\begin{equation}\notag
E_{b,c}{K_b\brack \lambda_b}={K_b; -1 \brack \lambda_b} E_{b,c};
\quad
E_{b,c}{K_c\brack \lambda_c}={K_c; 1 \brack \lambda_c} E_{b,c},
\end{equation}
(see, e.g., \cite[p.306]{tk}).
Thus, by induction, we have, for $M>0$,
\begin{equation}\label{rcom1}
E_{b,c}^{(M)}{K_b\brack \lambda_b}={K_b; -M \brack \lambda_b} E_{b,c}^{(M)};
\quad
E_{b,c}^{(M)}{K_c\brack \lambda_c}={K_c; M \brack \lambda_c} E_{b,c}^{(M)}.
\end{equation}
Hence, for $\lambda\in\Lambda(m|n,r), b\neq c,$
\begin{equation*}E_{b, c}^{(M)}{K\brack \lambda} =
E_{b, c}^{(M)}\prod_{a=1}^{m+n}{K_a\brack \lambda_a} 
={K_b; -M \brack \lambda_b}{K_c; M \brack \lambda_c}
\prod_{a\neq b,c} {K_a \brack \lambda_a}E_{b,c}^{(M)}.
\end{equation*}
Multiplying both sides on the left by ${K_b;0\brack \lambda_b+ M}$ and applying \eqref{rcom1}
yield in $U_{\up,\sZ}$:
\begin{equation*}
E_{b, c}^{(M)}{K_b;M\brack \lambda_b+ M}{K\brack \lambda} 
={K_b;0\brack \lambda_b+ M}{K_b; -M \brack \lambda_b}{K_c; M \brack \lambda_c}
\prod_{a\neq b,c} {K_a \brack \lambda_a}E_{b,c}^{(M)}.
\end{equation*}
We now compute the images of both sides in the quotient algebra ${\overline U}_{q,R}$:
\begin{equation}\notag
\mbox{LHS}=\sfe_{b, c}^{(M)}{\sfk_b;M\brack \lambda_b+ M} {1_\lambda}\\
=\sfe_{b, c}^{(M)}{\lambda_b+M\brack \lambda_b+ M}_q {1_\lambda}
=\sfe_{b, c}^{(M)}{1_\lambda}
\end{equation}
by (2), and
\begin{eqnarray*}
\text{RHS}&=&
{\sfk_b;0\brack \lambda_b+ M} {\sfk_b; -M \brack \lambda_b} {\sfk_c; M \brack \lambda_c} 
\prod_{a\neq b,c} {\sfk_a \brack \lambda_a} \sfe_{b,c}^{(M)}\\
&= &\sum_{\mu\in\Lambda(m|n,r)}\left(
{\sfk_b;0\brack \lambda_b+ M}  {\sfk_b;-M \brack \lambda_b} {\sfk_c; M \brack \lambda_c} 
\prod_{a\neq b,c} {\sfk_a \brack \lambda_a}  {1_\mu} \right) \sfe_{b,c}^{(M)}\quad
\\
&=&
\sum_{\mu\in\Lambda(m|n,r)}\left(
{\mu _b \brack \lambda_b+M}_{q}{\mu _b-M \brack \lambda_b}_{q}{\mu _c+M \brack \lambda_c}_{q}
\prod_{a\neq b,c} {\mu _a \brack \lambda_a}_{q} {1_\mu} \right) \sfe_{b,c}^{(M)}.
\end{eqnarray*}
Since, for $\lambda,\mu\in\Lambda(m|n,r),$ 
\begin{eqnarray*}
&&{\mu _b \brack \lambda_b+M}_{q} {\mu _b-M \brack \lambda_b}_{q} {\mu _c+M \brack \lambda_c}_{q} 
\prod_{a\neq b,c} {\mu _a \brack \lambda_a}_{q}  \neq 0\\
&&\iff
{\mu _b\geq\lambda_b+M},{\mu _c+M \geq \lambda_c},
 {\mu _a \geq \lambda_a} , \ \mbox{ for }{a\neq b,c}\\
 &&\iff \mu=\lambda+M\alpha,
\end{eqnarray*}
it follows that
 $$
\mbox{ RHS }  = \left\{ \begin{array}{lll}
              {1_{\lambda+M\alpha}}\sfe_{b, c}^{(M)}, & \mbox{  if  }  \lambda+M\alpha\in\Lambda(m|n,r),\\
            0,                              &   \mbox{ otherwise  }.
           \end{array}
         \right.
$$
as desired. The other case can be done similarly.
\end{proof}

\begin{remark}
The proof above is a modification of that of \cite[Proposition 3.7.1]{tk}. It works now over an arbitrary commutative ring and parameter $q\in R$.
\end{remark}

We are now ready to give a presentation for $S_{q,R}(m|n,r)$; compare the presentation over $\mathbb Q(\up)$ in \cite{tk}. Recall the map $\eta_{r,R}$ in \eqref{etaR} and Remark \ref{Lu90} for a presentation of $U_{q,R}(m|n)$.
%
%
%

\begin{theorem}\label{kfie}
For any commutative ring $R$,  the kernel of $\eta_{r, R}$ is the ideal $J_{r,R}$ generated by
the elements in \eqref{J_r}. In particular, the $q$-Schur superalgebra $S_{q,R}(m|n,r)$ can be presented by the generators as given in \eqref{Sgenerators} and relations for $U_{q,R}(m|n)$ together with  {\rm (1)--(2)} in Lemma \ref{reab}.
\end{theorem} 
\begin{proof}Recall the ideal $I_r$ of $U_{\up}(m|n)$ (over $\mathbb Q(\up)$) generated  by the elements in \eqref{ETK}. Let $J_{r,\sZ}$ be the ideal of $U_{\up,\sZ}(m|n)$ when $R=\sZ$. Base change to $\mathbb Q(\up)$ gives an ideal $J_{r,\mathbb Q(\up)}$ of $U_{\up}(m|n)$. Lemma \ref{reab} (3)\&(4) shows that $I_r\subseteq J_{r,\mathbb Q(\up)}$. On the other hand,
by \cite[Propositions 3.6.1--2]{tk}
these elements in \eqref{J_r}, when regarded as elements in $U_{\up,\sZ}$, are all in $I_r$. Hence, $I_r=J_{r,\mathbb Q(\up)}$. In particular, $J_{r,\sZ}\subseteq I_r\cap U_{\up,\sZ}$. Thus, there is an algebra epimorphism $\bar\pi:\overline{U}_{\up,\sZ}=U_{\up,\sZ}(m|n)/J_{r,\sZ}\to U_{\up,\sZ}/(I_r\cap U_{\up,\sZ})$ with $\bar\pi(\sfe_i)=\tte_i$ etc.. The latter is isomorphic to the image of $U_{\up,\sZ}(m|n)$ in $\overline{U}_{\up,\mathbb Q(\up)}=U_{\up}(m|n)/I_r$.

Now the proof of in  \cite[Theorem 3.12.1]{tk}, especially that given in \cite[Proposition 9.1]{doty2}, shows that the set ${\bf Y}$ in \eqref{ibas} forms a spanning set for $U_{\up,\sZ}/(I_r\cap U_{\up,\sZ})$. Similarly, by replacing $\tte_i, \ttf_i$ etc. by $\sfe_i,\sff_i$ etc., one constructs by  Lemma \ref{reab} a spanning set $\widetilde {\bf Y}$ for $\overline{U}_{\up,\sZ}$. 
Since $\overline{U}_{\up,\sZ}\otimes R\cong \overline{U}_{\up,R}$, it follows that $\widetilde{\bf Y}_R$ spans $\overline{U}_{q,R}$.

On the other hand, since the elements in \eqref{J_r} are all in the kernel of $\eta_{r,\sZ}$, it follows that $J_{r,R} \subseteq \ker\eta_{r,R}$, where $\eta_{r,R}$ is the epimorphism given in 
$\eqref{etaR}$. Consequently, $\eta_{r,R}$ induces an epimorphism $\bar\eta_{r,R}:{\overline U}_{q,R}={ U}_{q,R}/J_{r,R} \rightarrow  { S}_{q,R}.$ Hence, the image $\bar\eta_{r,R}(\widetilde{\bf  Y})$ spans $S_{q,R}$. Since $S_{q,R}$ is $R$-free of rank $|\widetilde{\bf  Y}|$, the transition matrix from $\bar\eta_{r,R}(\widetilde{\bf Y}_R) $ to a basis for $S_{q,R}$ must be invertible. This forces $\bar\eta_{r,R}(\widetilde{\bf Y}_R)$ is linearly independent. Therefore, $\bar\eta_{r,R}$ must be an isomorphism.
\end{proof}

\begin{remark} Both proofs in \cite{tk} and \cite{doty2} for the fact that $\bf Y$ spans $U_{\up,\sZ}/I_r\cap U_{\up,\sZ}$ and $\widetilde{\bf Y}$ spans $\overline{U}_{\up,\sZ}$ use a PBW type basis involving all root vectors. Thus, almost all the commutation formulas in Propositions \ref{ppco} and \ref{cqrv} have to be used in a lengthy case-by-case argument. However, if we use a monomial basis in the divided powers of generators as given in \cite{dp}, the number of cases can be reduced significantly and a complete proof can be seen easily.
\end{remark}



 The following result is a super version of  \cite[Thm 9.3]{dp}.

\begin{corollary}\label{qq-1}
 Assume $m,n\geq r$ and let $\omega\in\La(m|n,r)$ be of the form
$$\omega=(0^a,1^r,0^{m-a-r}|{\bf 0})\text{ or }\omega=({\bf 0}|0^a,1^r,0^{n-a-r}).$$ Then the elements $\ttC_i=1_\omega \tte_{i}\ttf_{i}1_\omega$ in $S_{q,R}(m|n,r)$ for $\omega_i=1$ satisfy the relations 
$$\ttC_i^2=(q^{-1}+q)\ttC_i,\;\; \ttC_i\ttC_j= \ttC_j\ttC_i (|i-j|>1),\;\; \ttC_i\ttC_{i+1}\ttC_i-\ttC_i=\ttC_{i+1}\ttC_{i}\ttC_{i+1}-\ttC_{i+1}.$$
In particular, there is an $R$-algebra isomorphism
$1_\omega S_{q,R}(m|n,r)1_\omega\cong H_{q^2,R}(r)$.
\end{corollary}
\begin{proof} We may simply modify the proof of \cite[Thm 9.3]{dp} to prove these relations. 
To see the last assertion, let $t_{i}=q\ttC_{i}-1_\omega$  for all $i$ with $\omega_i=1$.
Then $\{t_{i}\mid i\in[1,m+n],\omega_i=1\}$ generate a subalgebra isomorphic to $H_{q^2,R}(r)$ under the map $t_{a+i}\mapsto T_{i}$ (cf. the proof for \cite[Thm 9.3]{dp}).\footnote{If we put $t_{i}'=q^{-1}\ttC_{i}-1$, then the relations for $\ttC_i$ are equivalent to 
$(t'_i)^2=(q^{-2}-1)t'_i+q^{-2},\;\;t'_it'_{i'}=t'_{i'}t'_i,\;\; t'_jt'_{j'}t'_j=t'_{j'}t'_jt'_{j'},$
where $|i-i'|>1$ and $|j-j'|=1$. Then $t'_i$ generate a subalgebra isomorphic to $H_{q^{-2},R}(r)$. }
\end{proof}

\begin{lemma}\label{sigS}
The isomorphism  $\sigma$ considered in \eqref{sigma} and \eqref{intsig} induces  an
$R$-algebra isomorphisms $
\sigma_R:U_{q,R}(\mathfrak{gl}_{n|n})\longrightarrow U_{q,R}(\mathfrak{gl}_{n|n}),$ which further induces an algebra automorphism, by abuse of notation,
$$\sigma_R:S_{q,R}(n|n,r)\longrightarrow S_{q,R}(n|n,r).$$
Moreover, if $m=n\geq r$ and 
$
\omega=(1^r,0^{n-r}|0^n),
\omega'=(0^n|0^{n-r},1^r)\in \Lambda(n|n,r)
$, then the automorphism $\sigma$ restricts to an $R$-algebra isomorphism
\begin{equation}\label{Bsig}
\bar\sigma:1_\omega S_{q,R}(n|n,r)1_\omega\longrightarrow 1_{\omega'} S_{q,R}(n|n,r)1_{\omega'},\;\; t_i\mapsto t_{2n-i} \;\;(1\leq i\leq r-1).
\end{equation}
\end{lemma}
\begin{proof}The first automorphism is induced from \eqref{intsig}, while the second is clear since $\sigma(\ker\eta_{r,R})= \ker\eta_{r,R}$ by  Theorem \ref{kfie}. The last assertion follows easily from the fact that $\sigma(1_\omega)=1_{\omega'}$ and
$1_\omega \tte_{i}\ttf_{i}1_\omega=1_\omega\ttf_{i}\tte_{i}1_\omega$ since
$\frac{\ttk_i\ttk_{i+1}^{-1}-\ttk_i^{-1}\ttk_{i+1}}{\up_i-\up_i^{-1}}1_\omega=0$.
\end{proof}
This result shows that the automorphism given in \eqref{barsig} agrees with the automorhism  $\sigma$ above.

%
%
%
\section{Finite dimensional weight supermodules of $U_{q,F}(m|n)$}
{\it From now on, we will assume $R=F$ is a field of characteristic $\neq2$ and $q\in F$ is an $l'$th primitive root of unity with $l'\geq3$. By setting $l'=\infty$, we may also include the case where $q$ is not a root of unity.} 
We first describe a classification of the irreducible weight supermodules of ${ U}_{q,F}=U_{q,F}(m|n)$ by their highest weights. We then use the result in the previous section to give a criterion for polynomial weight $U_{q,F}$-supermodules.
 
For a  ${ U}_{q,F}$-supermodule $V$ and
$\lambda\in \mathbb Z^{m+n}$, define its (nonzero)
$\la$-weight space (of type {\bf 1}) by
\begin{equation}\label{wt space}
V_\lambda=
\{ v\in V \,|\, K_a.v=q_a^{\lambda_a}v, \ {K_a;c\brack
t}.v={\lambda_a+c\brack t}_qv ,\  \forall \, t\in\mathbb N, c\in
\mathbb Z \}.
\end{equation}
Note that, by using ${{K_{b}^{-1}; c}\brack   t}= (-1)^t{{K_{b}; -(c+1)+t} \brack   t}$
and \cite[Lemma 14.18]{ddp}, we deduce from \eqref{wt space} that, for $v\in V_\la$,  
\begin{equation}\label{wt space1}
{K_a^{-1};c\brack t}.v={-\la_a+c\brack t}_q.v,\quad {K_{a,b};c\brack t}.v={\la_a-(-1)^{\bar a+\bar b}\la_b+c\brack t}_q.v
\end{equation}
For example, as a quotient of $U_{q,F}$, $V=S_{q,F}$ is a ${ U}_{q,F}$-supermodule. By Lemma \ref{reab}(2), the $\la$-weight space $V_\la=1_\la S_{q,F}$.

Define the partial ordering $\leq$ on $\mathbb Z^{m+n}$ by setting $\mu\leq\lambda$
if and only if $\lambda-\mu$ is a nonnegative sum of simple roots $\alpha_i.$ Denote $\wt(v)=\lambda $ for  $v\in V_\lambda$ and let $\wt(v)_h$ be the $h$th component of $\wt(v)$.

\begin{lemma}
For a ${ U}_{q,F}$-supermodule $V$, $\lambda\in  \mathbb Z^{m+n}$, and $\alpha\in\Phi$, we have
\begin{equation}\label{wei}
E_{\alpha}^{(M)} V_\lambda \subseteq V_{\lambda+M\alpha}.
\end{equation}
 In particular, for $0\neq v\in V_\lambda$ 
and $h\in[1,m+n)$, if $E_{b,a}^{(M)}.v\neq0$ for some 
$a<b$ and $M>0$, then
\begin{equation}\label{cmax}
\wt(v)>\wt(E_{b,a}^{(M)}.v),\quad \wt(v)_h\leq\wt(E_{b,a}^{(M)}.v)_h\;\; (\forall h\neq a).
\end{equation}
\end{lemma}
\begin{proof}The first assertion follows easily from (\ref{rcom1}). Suppose now $E_{b,a}^{(M)}.v\neq 0$.
By \eqref{wei}, 
$\wt(E_{b,a}^{(M)}.v)=\wt(v)-M({\epsilon_a}-{\epsilon_{b}})<\wt(v),$
and $\wt(E_{b,a}^{(M)}.v)_h\geq \wt(v)_h$ whenever $h\neq a$.
\end{proof}

Call    $\lambda=\sum_{i=1}^{m+n}\lambda_i{ \epsilon_i}$ ($\lambda_i\in \mathbb Z$)
 a {\it weight} of $V$ if $V_\la\neq0$ and denote by $\pi(V)=\{\mu\in \mathbb Z^{m+n} \,|\, V_\mu\neq 0\}$ the set of weights of $V$. By \eqref{wei}, $\bigoplus_{\lambda\in \pi(V)} V_\lambda$ is a submodule of $V$. If
$V=\bigoplus_{\lambda\in \pi(V)} V_\lambda$, we call $V$ a {\it weight supermodule} (of type {\bf 1}).  
For example, the natural supermodule $V(m|n)$ and its tensor product $V(m|n)^{\otimes r}$ are weight supermodules with 
$\pi(V(m|n)^{\otimes r})= \Lambda(m|n,r)$. 

For every weight supermodule $V=\bigoplus_{\lambda\in \pi(V)} V_\lambda$, we may change its superspace structure to get a standard one $V^\sfs$ associated with $V$, where 
$$V^\sfs=V \text{ as a $U_{q,F}$-module, but }V^\sfs_{i} =\bigoplus_{\overline{|\mu^{(1)}|}\equiv i}V_\mu\;(i\in\mathbb Z_2).$$
Clearly, $V^\sfs$ is a weight supermodule.

\begin{lemma}\label{V^s} For every weight $U_{q,F}$-supermodule $V$, there is a supermodule isomorphism $V\cong V^\sfs$.
\end{lemma}
\begin{proof} Since the parity function $\delta_V:v\mapsto(-1)^{\bar v}v$ on the superspace $V$ stabilises every weight space $V_\mu$, it follows that $V_\mu=(V_\mu)_{\bar0}\oplus (V_\mu)_{\bar1}$. Putting $\pi(V)_i=\{\mu\in\pi(V)\mid \overline{|\mu^{(1)}|}=i\}$, we have
$$V=\bigoplus_{\mu\in\pi(V)}\big((V_\mu)_{\bar0}\oplus (V_\mu)_{\bar1}\big)=V_1\oplus V_2,$$
where $V_1=\bigoplus_{\mu\in\pi(V)_{\bar0}}(V_\mu)_{\bar0}\oplus \bigoplus_{\mu\in\pi(V)_{\bar1}}(V_\mu)_{\bar1}$ and $V_2=\bigoplus_{\mu\in\pi(V)_{\bar0}}(V_\mu)_{\bar1}\oplus \bigoplus_{\mu\in\pi(V)_{\bar1}}(V_\mu)_{\bar0}$.
Clearly, both $V_1$ and $V_2$ are subsupermodules and $V_1\cong V_1^\sfs$ and $V_2\cong \Pi(V_2^\sfs)$, where $\Pi$ is the parity functor. Thus, $V\cong V_1\oplus \Pi(V_2)\cong V_1^\sfs\oplus V_2^\sfs=V^\sfs$, as desired.
\end{proof}


We call a nonzero weight vector $\fkm_\lambda$  a {\it maximal vector} if it
satisfies
$$
E_{i}^{(M)}.\fkm_\lambda=0,\quad   \mbox { for }1\leq i\leq m+n-1\mbox { and } M>0.
$$
Call $V$ a {\it highest weight supermodule} if it is generated by a maximal vector.

Let\footnote{The set $\mathbb Z^{m|n}_\pp$ is denoted by $X^+(T)$ in \cite[p.23]{bru}. Also, the notation $\Lambda^\pp(m|n,r)$ there has a different meaning; see footnote 6 below.}
\begin{equation}\label{La++}
\aligned
\mathbb Z^{m|n}_\pp&=\{\la\in\mathbb Z^{m+n}\mid \lambda_1\geq\cdots\geq\lambda_m,\lambda_{m+1}\geq\cdots\geq\lambda_{m+n}\},\\
 \Lambda^\pp(m|n)&= \Lambda(m|n)\cap\mathbb Z^{m|n}_\pp,\quad \Lambda^\pp(m|n,r)= \Lambda^\pp(m|n)\cap\Lambda(m|n,r).\\
  \endaligned
  \end{equation} 
For $\la\in \mathbb Z^{m|n}_\pp$, let $F_\la$ be a one-dimensional $U_{q,F}^0$-module of weight $\la$. By inflating $F_\la$ to a ${U}_{q,F}^{\geq0}$-supermodule and then inducing to $U_{q,F}$, we obtain the induced supermodule or Verma supermodule  $Y(\lambda)={ U}_{q,F}\otimes_{{U}_{q,F}^{\geq0}} F_\lambda.$ 

Similar to the non-super case (see \cite[\S\S6.1,6.2]{Lu}), 
$Y(\lambda)$ is a highest weight supermodule with highest
weight $\lambda,$ and $\dim Y(\la)_\la=1$. Thus, every proper submodule of $Y(\la)$ is contained in the subspace $\bigoplus_{\mu<\la}Y(\la)_\mu$. Hence, $Y(\la)$ has a unique maximal super submodule and hence, a unique irreducible quotient $L(\la)$.

We remark that the irreducible supermodule $L(\la)$ can also be constructed through Kac modules, cf.  \cite[Section III]{zrb}\footnote{Note that the quantum supergroup at a root of unity in \cite{zrb} is not the quantum hyperalgebra $U_{q, F}(m|n)$ here. Compare \cite[\S\S11.1,11.2]{cp}.}
and \cite[Thm 4.6]{bru}.
\begin{proposition}\label{ciir}In the category $U_{q,F}$-{\bf mod} of finite dimensional weight $U_{q,F}(m|n)$-supermodules,
the set $\{L(\lambda)\mid \lambda\in  \mathbb Z^{m|n}_\pp\}$
forms a complete set of irreducible objects.
\end{proposition}
\begin{proof}Suppose $L$ is an irreducible weight $U_{q,F}$-supermodule. Then $L=\bigoplus_{\mu\in\pi(L)} L_\mu$. Let $\la\in\pi(L)$ be a maximal weight and $0\neq v\in L_\la$. We must have $L=U_{q,F}v$. Hence, $L$ is a highest weight module. Thus, if $U_{\bar 0}\cong U_{q,F}(\mathfrak{gl}_m\oplus\mathfrak{gl}_n)$ is the even subalgebra of $U_{q,F}(m|n)$, then $U_{\bar 0}v$ is a highest weight module. Hence, it has a unique irreducible quotient of highest weight $\la$. Hence, $\la \in  \mathbb Z^{m|n}_\pp$. This completes the proof.
\end{proof}

{\it From now on, unless otherwise stated, we assume every weight supermodule is finite dimensional.}
Let $V=\bigoplus_{\lambda\in \mathbb Z^{m+n}} V_\lambda$ be a finite dimensional  weight $U_{q,F}$-supermodule. Then, for any $r\in\mathbb Z$,
$V_r=\oplus_{|\mu|=r} V_\mu$ is a subsupermodule by \eqref{wei}. If $V=V_r$, then $V$ is called a $U_{q,F}$-supermodule of {\it degree} $r$.
We call $V$  a  {\it polynomial supermodule} of ${ U}_{q,F}$ if $V$ is a weight module with
$\pi(V)\subseteq \Lambda(m|n).$    Since, for a polynomial supermodule $V$ we have
$V=\oplus_{r\geq 0} V_r,$ we need only consider $V_r,$
which is called a polynomial supermodule of degree $r$. Unlike the nonsuper case, we will see in the next section that only a subset of $\Lambda^\pp(m|n)$ labels all polynomial irreducible $U_{q,F}$-supermodules.

We now describe a vanishing ideal for all polynomial $U_{q,F}$-supermodules of degree $r$.
Recall the algebra epimorphism $\eta_{r,F}$ in \eqref{etaR} and its kernel in Theorem \ref{kfie}.

%
%
\begin{proposition}\label{cpol} Let $V$ a polynomial $U_{q,F}$-supermodule and $r>0$. Then
$V=V_r$ if and only if
$\ker(\eta_{r,F}).V= 0$. Hence, every polynomial $U_{q,F}(m|n)$-supermodule of degree $r$ is an inflation of a $S_{q,F}(m|n,r)$-supermodule.
\end{proposition}
\begin{proof}
The sufficiency is clear since every ${ S}_{q,F}$-supermodule has its weights in $\La(m|n,r)$. Suppose now $V$ is a polynomial $U_{q,F}$-supermodule of degree $r$.
To prove  $\ker(\eta_{r,F}).V= 0,$ by Theorem \ref{kfie},
it is sufficient to verify  every element in \eqref{J_r} vanishes $V$.

Choose $0\neq m_\mu\in V_\mu $  with  $\mu\in \Lambda(m|n,r).$ Then, for $a\in [1,m+n]$,
\begin{eqnarray*}
(1)&&(1-\sum_{\lambda\in\Lambda(m|n,r)} {K\brack\lambda}). m_\mu
=(1-\sum_{\lambda\in\Lambda(m|n,r)} {\mu\brack \lambda}_q). m_\mu
\quad (\mbox{Recall }  {\mu\brack \lambda}_q=\prod_{i=1}^{m+n} {\mu_i\brack \lambda_i}_q.)\\
&&=(1- {\lambda\brack \lambda}_q). m_\mu=0
\qquad\qquad (\mbox{since } {\mu\brack \lambda}_q=\delta_{\mu\lambda}).
\end{eqnarray*}
\begin{eqnarray*}
(2)&&(K_a^{\pm 1} {K\brack\lambda}-q_a^{\pm \lambda_a}{K\brack\lambda}). m_\mu
 =(K_a^{\pm 1} -q_a^{\pm \lambda_a}){K\brack\lambda}. m_\mu
=(K_a^{\pm 1} -q_a^{\pm \lambda_a}){\mu\brack \lambda}_q. m_\mu\\
&&= (q_a^{\pm \mu_a} -q_a^{\pm \lambda_a}){\mu\brack \lambda}_q. m_\mu  
= 0 \qquad\qquad (\mbox{since } {\mu\brack \lambda}_q=\delta_{\mu\lambda}).
\end{eqnarray*}
The proof of $({K_a;c\brack t}{K\brack\lambda} -{\lambda_a+c\brack t}_q {K\brack\lambda}).m_\mu=0$
is similar as (2).
\end{proof}
\begin{remark}
By this proposition,  the full subcategory of finite dimensional polynomial $U_{q,F}(m|n)$-supermodules of degree $r$  is equivalent to the
category of finite dimensional supermodules over the $q$-Schur  superalgebra $
S_{q,F}(m|n,r).$

Rui and the first author (\cite[Thm 9.8]{dr}) showed that
$
{ S}_{q,F}(m|n,r) \cong  { A}_F(m|n,r)^*,
$
where ${ A}_F(m|n,r)$ is the $r$th homogenous component of the quantum matrix
superalgebra. 
Hence, one may also follow Green's original definition in \cite{Gr} to define polynomial $U_{q,F}(m|n)$-supermodules through $A_F(m|n,r)$-cosupermodules.
\end{remark}


%
%
%
%
%
%

%
%
%
\section{Polynomial irreducible ${ U}_{q,F}(m|n)$-supermodules}\label{sec}
Throughout the section, $F$ denotes a field of characteristic $\neq2$ and $q\in F$. If $q$ is an $l'$th primitive root of
unity in $F$, let 
$$l=\begin{cases} l',&\text{ if $l'$ is odd};\\
\frac{l'}2,&\text{ if $l'$ is even.}
\end{cases}$$
In this case, $q^2$ is an $l$-th primitive root of unity. If $q$ is not a root of unity, then we set $l=\infty$.
As before, we use the abbreviation $U_{q,F},S_{q,F}$ for $U_{q,F}(m|n),S_{q,F}(m|n,r)$, respectively.


Let   $\Lambda^+$ be the set of all partitions.  Following  \cite[Lem. 6.2]{bru} or \cite[Lem. 1]{xu}, we first define the map
\begin{equation}\label{j_l}
j_l: \  \Lambda^+ \rightarrow \mathbb N.
\end{equation}

For $\lambda\in  \Lambda^+$ of length $\ell(\la)=d$ (i.e., $\lambda_{d+1}=0$ and $\lambda_d\neq0$), let $x_{d+1}=x_{d+2}=\cdots =0$ and define $x_d, x_{d-1},\cdots x_1\in \{0,1\}$ recursively
by setting
\begin{equation}\label{x_i}
x_i= \left\{ \begin{array}{lll}
             1, & \mbox{  if  }  \lambda_i+x_{i+1}+x_{i+2}+\cdots \not\equiv 0\quad   (\!\!\!\!\!\mod l);\\
            0, &   \mbox{  if  }  \lambda_i+x_{i+1}+x_{i+2}+\cdots \equiv  0\quad   (\!\!\!\!\!\mod l).
           \end{array}
         \right.
\end{equation}
Let
$$j_l(\lambda)=x_1+x_2+\cdots=\sum_{i=1}^d x_i.$$
Then $j_l(\lambda)\leq d$.
It is clear from the definition that, if $l=\infty$ (i.e., $q$ is not a root of unity), then $j_\infty(\lambda)=d$.
In general, there exists subsequence $1\leq i_1<\cdots<i_t\leq d$, where $t=j_l(\lambda)$, such that
$x_{i_1}+\cdots+x_{i_t}=t$ and $j_l(\lambda_{i_1},\ldots,\lambda_{i_t})=t$.

The map $j_l$ is closely related to Xu's algorithm for computing the Mullineux map; see \S8 below, \cite[\S6]{bru}, and Remark \ref{Xu}(2).

\begin{lemma}\label{mull}For $\lambda\in  \La^+,$
if $1\leq i_1< i_2\cdots <i_t\leq\ell(\la), $
then
$j_l(\lambda_{i_1},\lambda_{i_2},\cdots, \lambda_{i_t})\leq j_l(\lambda).$
In particular,
$j_l(\lambda_{i_1},\lambda_{i_2},\cdots, \lambda_{i_t})=t$  if and only if
$$
\prod_{s=1}^{t}{\lambda_{i_{s}}+t-s \brack 1}_q\neq 0.
$$
\end{lemma}
\begin{proof}Let $\mu=(\lambda_{i_1},\lambda_{i_2},\cdots, \lambda_{i_t})$.
By the observation above, we may  assume  that $j_l(\mu)=t$ and prove $t\leq j_l(\lambda)$.
Define $y_1, y_2,\ldots,y_t,\ldots \in \{0,1\}$ for $\mu$, similar to the $x_i$ for $\lambda$, as above.   Since $j_l(\mu)= t$,  we have $y_t=y_{t-1}=\cdots =y_1=1$ and $y_i=0$ for all $i>t$. Thus, by definition, we have, for $s=t,t-1,\ldots,1$,
$$\lambda_{i_s}+t-s=\mu_{s}+t-s \not\equiv 0\quad   (\!\!\!\!\!\mod l).$$
We claim that there exist $1\leq i_1'<\cdots<i_t'\leq d=\ell(\la)$ such that $x_{i_s'}=1$. Indeed,
let $i_t'\in[1,d]$  be the maximal index such that $\lambda_{i_t'}\not\equiv  0({\rm mod}\,l)$. Then $i_t'\in [i_t,d]$ and $ x_{i_t'}=1.$
For every $s=t-1, t-2,\ldots 1$, the above congruence relations guarantee that there exists $i_s'\in [i_{s}, i_{s+1}') $,  maximal  in the interval,
such that
$
\lambda_{i_{s}'}+t-s\not\equiv 0({\rm mod}\, l).
$
Thus, by the selection of $i_s'$, we have $x_{i_{s}'}=1$ for all $s\in[1,t]$, proving the claim and,
hence, the first assertion.

Since  ${\lambda_{i_{s}}+t-s \brack 1}_q=[\lambda_{i_{s}}+t-s]_q\neq 0
$  if and only if 
$\lambda_{i_{s}}+t-s \not\equiv 0({\rm mod}\, l)$ (noting $q^2$ is a primitive $l$th root of unity),
 the last assertion is clear.
\end{proof}

%
%
%


 For  $\lambda\in \Lambda^\pp(m|n,r)$, define the ``modulo $l$'' subset
\begin{equation}\label{++l}
  \Lambda^\pp_l(m|n,r)=
  \{\lambda\in \Lambda^\pp(m|n,r)
  \,
  |
  \,
  j_l(\la^{(1)})\leq \lambda_m
  \}.
\end{equation}
If $\l=\infty$, then $\lambda^{(0)}$ concatenating with the dual of $\lambda^{(1)}$ is a well defined partition of $r$ 
and the set $ \Lambda^\pp_\infty(m|n,r)$ is identified with
\begin{equation}\label{spartition}
\Lambda^+(r)_{m|n}:=\{\la\in\Lambda^+\mid \lambda_{m+1}\leq n,|\lambda|=r\}
\end{equation}
(see \cite[(4.0.2)]{dr}). This set is used in \cite{dr} to label irreducible $S_{\up,\mathbb Q(\up)}(m|n,r)$-modules.

\begin{remark}If $l=p$ is a prime,  $\Lambda^\pp_p(m|n,r)$ is used in \cite{bru}\footnote{This set is denoted by $ \Lambda^\pp(m|n,r)$ in \cite[Thm 6.5]{bru}, where $\la^{(1)}$ is called the tail $t(\la)$ of $\la$.} to parametrize the irreducible supermodules of the
Schur superalgebra $S(m|n,r)$ in positive characteristic $p$. In the theorem below, we will generalise this result to  the quantum Schur superalgebras at every primitive $l'$-th root of unity $q$.
\end{remark}

%
%
%
\begin{lemma}\label{comp}
For $\lambda\in \Lambda^\pp(m|n,r),$ let $\fkm_\lambda$ be a maximal vector of weight $\lambda$ for a weight $U_{q,F}$-supermodule. 
Consider the sequences $1\leq h< i_1<i_2< \cdots <  i_{{s}}\leq m+n$ and
$(a_1,\cdots, a_s)\in \mathbb Z_{>0}^s$, where $a_t=1$ whenever $\alpha_{h,i_t}$ is an odd root. 
\begin{itemize}
\item[(1)] If $1\leq t\leq s$ and $i>i_t$, then
$ E_{h,i}^{(b)}  .(E_{{i_{{t}}},h}^{(a_t)} \cdots E_{i_1, h}^{(a_1)}.\fkm_\lambda)
=0, \forall    b>0.$
\item[(2)]  $(E_{h,i_1}^{(a_1)}\cdots  E_{h,{i_{{s}}}}^{(a_s)}  )
.(E_{{i_{{s}}},h}^{(a_s)} \cdots E_{i_1, h}^{(a_1)}.\fkm_\lambda)
=\displaystyle{\prod_{t=1}^{{s}}
 {\lambda_{h}-(-1)^{{\bar h}+{\bar i_t}}\lambda_{i_t}-a_{t-1}\cdots-a_1 \brack a_t}_q \fkm_\lambda.}$
\end{itemize}
\end{lemma}
\begin{proof}
Applying the anti-automorphism $\Upsilon$ defined in \eqref{Up} to the formulas
in Proposition \ref{cqrv}(3) yields the following commutation formulas in $U_{q,F}$:
$$E_{h,i}^{(b)}  E_{{i_{t},h}}^{(a_t)} 
=\begin{cases}
(-1)^{\bar E_{h,i}\bar E_{i_t,h}}E_{i_t,h}E_{h,i}-(-1)^{\bar E_{h,i}\bar E_{i_t,h}}E_{i_t,i}K_{i_t,h},&\text{if } b=a_t=1;\\
\displaystyle\sum_{k=0}^{\min({b,a_t})} (-1)^kq_{i_t}^{k(a_{t}-1-k)}
E_{{i_{t}},h}^{(a_t-k)} K_{h,i_t}^{-k} E_{h,i}^{(b-k)} E_{i_t,i}^{(k)},&\text{otherwise.}
\end{cases}$$
Since either $b-k>0$ or $k>0$ and $\fkm_\la$ is a maximal vector, assertion (1) is clear if $t=1$. The general case follows from induction.


We now prove (2). By 
Proposition \ref{cqrv}(4) and assertion (1), if $a_s>1$,
\begin{equation}\notag
\begin{aligned}
 & E_{h,{i_{{s}}}}^{(a_s)}  
.(E_{{i_{{s}}},h}^{(a_s)} \cdots E_{i_1, h}^{(a_1)}.\fkm_\lambda)\\
=&\left(\sum_{t=0}^{a_s} E_{{i_{{s}}},h}^{(a_s-t)}
\begin{bmatrix} K_{h,i_s}; 2t-a_s-a_s \\ t \end{bmatrix} E_{h,{i_{{s}}}}^{(a_s-t)}\right)
(E_{{i_{{s-1}}},h}^{(a_{s-1})} \cdots E_{i_1, h}^{(a_1)}.\fkm_\lambda)\\
=&\begin{bmatrix} K_{h,i_s}; 0 \\   a_s \end{bmatrix} 
(E_{{i_{{s-1}}},h}^{(a_{s-1})} \cdots E_{i_1, h}^{(a_1)}.\fkm_\lambda) \quad(\text{by } (1))\\
=& {\lambda_{h}-a_{s-1}\cdots-a_1-(-1)^{{\bar h}+{\bar i_s}}\lambda_{i_s} \brack a_s}_q
(E_{{i_{{s-1}}},h}^{(a_{s-1})} \cdots E_{i_1, h}^{(a_1)}.\fkm_\lambda),\notag
\end{aligned}
\end{equation}
by \eqref{wei} and \eqref{wt space1}. The $a_s=1$ case is similar. Induction on $s$ proves (2).
\end{proof}


%
%
%
\begin{theorem}\label{cirr} 
For $r>0$ and $\lambda\in\Lambda^\pp(m|n,r)$, the irreducible $U_{q,F}$-supermodule $L(\lambda)$ is polynomial if and only if
$\lambda\in\Lambda^\pp_l(m|n,r)$. In particular, 
the set
$\{L(\lambda)
\,
|
\,
\lambda\in \Lambda^\pp_l(m|n,r)\}
$
forms a compete set of all pairwise non-isomorphic polynomial
irreducible ${ U}_{q,F}$-supermodules of degree $r$ in $U_{q,F}$-{\bf mod}.
\end{theorem}

\begin{proof}
Choose a maximal vector  $0\neq \fkm_\lambda\in L(\lambda)_\lambda$  then
 \begin{equation}\notag
L(\lambda)= { U}_{F}. \fkm_\lambda=  { U}_{F}^-. \fkm_\lambda .
\end{equation}

First, we prove that, for $\lambda\in \Lambda^\pp(m|n,r)$, if  $j_l(\lambda^{(1)})> \lambda_m,$ then
$L(\lambda)$ is not a polynomial supermodule. Suppose that  $\lambda\notin \Lambda^\pp_l(m|n,r)$ and so
$s:= j(\lambda^{(1)})\geq \lambda_m+1\geq 1.$
For partition $\lambda^{(1)}=(\lambda_{m+1},\cdots,\lambda_{m+n}),$ the sequence $x_1,x_2,\ldots,x_d$ defined in \eqref{x_i} satisfies $\sum_{i\geq 1} x_i=s.$  Let $i_1<i_2< \cdots <  i_{\lambda_m+1}$ be the indices of the last 
 $\lambda_m+1$ nonzero terms in the sequence.
 Then $x_{i_t}=1$ for all $t\in [1,\lambda_m+1]$ and
$$
j_l( \lambda_{{m+i_{1}}},  \cdots, \lambda_{m+i_{\lambda_m+1}})={\lambda_m+1}.
$$
By Lemma \ref{mull}
\begin{equation}\notag
\prod_{t=1}^{\lambda_m+1}{\lambda_{m+i_{t}}+\lambda_m+1-t \brack 1}_q\neq 0.
\end{equation}
Thus,  Lemma \ref{comp} implies
\begin{eqnarray}\label{non}
&&(E_{m,m+i_1}\cdots E_{m,{m+i_{{\lambda_m+1}}}})
               (E_{m+{i_{{\lambda_m+1}}},m} \cdots E_{m+i_1,m} ).\fkm_\lambda \notag \\
&&=\prod_{t=1}^{{\lambda_m+1}}
                          {\lambda_{m}+\lambda_{m+i_t}-t+1 \brack 1}_q\fkm_\lambda
\neq 0.\qquad
\end{eqnarray}
Hence, $E_{m+{i_{{\lambda_m+1}}},m} \cdots E_{m+i_1,m}.\fkm_\lambda\neq0$, forcing
$L(\lambda)_{\lambda-\alpha_{m,m+i_1}-\cdots-\alpha_{m,m+{i_{{\lambda_m+1}}}}}\neq0. $
Since
$$\aligned
{\lambda-\alpha_{m,m+i_1}-\cdots-\alpha_{m,m+{i_{{\lambda_m+1}}}}}
&=\lambda-(\lambda_m+1)\epsilon_m+\epsilon_{m+i_1}+\cdots+\epsilon_{m+i_{\lambda_m+1}},
\endaligned$$
whose $m$th component is $-1$,
it follows that $\pi(L(\la))\not\subseteq \Lambda(m|n,r)$. Hence,
$L(\lambda)$ is not a polynomial supermodule.

%
%

We now prove the converse.
Suppose $L(\lambda)$  is not a polynomial supermodule of degree $r=|\lambda|$. Then
there exists  $ \nu\in\mathbb Z^{m+n}$ such that $|\nu|=r,$  $ L(\lambda)_\nu\neq 0,$ and  $ \nu_h<0$ for some
$h\in [1,m+n].$ In fact, we may assume that $h<m+n$. This can be seen from the fact that $L(\lambda)= { U}_{F}^-. \fkm_\lambda$
 is spanned by vectors $\prod_{a<b}E_{b,a}^{(A_{b,a})}.\fkm_\lambda$ whose weights are of the form $\lambda-\sum_{a<b} A_{b,a}(\epsilon_a-\epsilon_b)$. 
 We need to prove that $\lambda\not\in\Lambda^\pp_l(m|n,r)$.

Let
$$\mu=\max\{\nu\in \pi(L(\lambda))\, |\, \nu_h<0\}.$$

\noindent
{\bf Claim 1.} The weight space $L(\lambda)_\mu$ is spanned by the  vectors 
$$
\{ 
\prod_{h+1\leq b\leq m+n}
     E_{b,h}^{(A_{b,h})}\, .  \mathfrak{m}_\lambda 
\, |\,
  A_{b,h}\in\mathbb N  \mbox{ and  }\mu=\lambda-A_{h+1,h}\alpha_{h,h+1}-\cdots-A_{m+n,h}\alpha_{h,m+n}
\}.
$$
{\bf Proof of Claim 1.} Fix  an ordering on $\Phi^+$ such that the sequence ends with the $m+n-h$ positive roots:
$\alpha_{h,m+n},\alpha_{h,m+n-1},\ldots,\alpha_{h,h+1}$. Then
\begin{equation}\notag
L(\lambda)= { U}_{F}^-. \fkm_\lambda
=\span\bigg\{
\prod_{{\epsilon_a-\epsilon_b\in \Phi^+,a\neq h}}
E_{b,a}^{(A_{b,a})}\,
\prod_{h+1\leq b\leq m+n}E_{b,h}^{(A_{b,h})} .  \fkm_\lambda\bigg|A\in M(m|n)^-
 \bigg\}.
\end{equation}

For every nonzero spanning vector of the form,
\begin{equation}\notag
w=\prod_{{\epsilon_a-\epsilon_b\in\Phi^+}\atop a\neq h}
E_{b,a}^{(A_{b,a})}\bigg(
\prod_{h+1\leq b\leq m+n }
 E_{b,h}^{(A_{b,h})}\, . \mathfrak{m}_\lambda\bigg)
\end{equation}
satisfying $ \wt(w)_h<0$, if $ \prod_{{\epsilon_a-\epsilon_b\in\Phi^+}\atop a\neq h}
E_{b,a}^{(A_{b,a})}\neq 1,$
repeatedly applying  \eqref{cmax} yields
\begin{equation}\notag
\wt(
\prod_{h+1\leq b\leq m+n }
 E_{b,h}^{(A_{b,h})}\, . \mathfrak{m}_\lambda)
>
\wt(w).
\end{equation}
and
\begin{equation}\notag
\left(\wt(
\prod_{h+1\leq b\leq m+n }
 E_{b,h}^{(A_{b,h})}\, . \mathfrak{m}_\lambda)\right)_h
 \leq \wt(w)
_h<0
\end{equation}
Thus, if $w\in L(\lambda)_\mu$, then the maximality of $\mu$ forces $ \prod_{{\epsilon_a-\epsilon_b\in\Phi^+}\atop a\neq h}
E_{b,a}^{(A_{b,a})}=1$, proving Claim 1.

By Claim 1, we  choose a nonzero vector $v\in L(\lambda)_\mu$ of the form:
\begin{equation}\label{maxmu}
v= E_{{i_{{s}}},h}^{(a_s)}  E_{{i_{{s-1}}},h}^{(a_{s-1})} \cdots E_{i_1, h}^{(a_1)}.\mathfrak{m}_\lambda\neq0,
\end{equation}
for some sequences $h< i_1<i_2< \cdots <  i_{{s}}\leq m+n$ and
$(a_1,\cdots, a_s)\in (\mathbb Z_{>0})^s$ where $a_t=1$ whenever
 ${\alpha}_{h,i_t}$ is an odd root. Then $\mu=\wt(v)=\la+\sum_{t=1}^sa_t(\epsilon_{i_t}-\epsilon_h)$.

Since $\wt(E_{{i_{{s-1}}},h}^{(a_{s-1})} \cdots E_{i_1, h}^{(a_1)}.\mathfrak{m}_\lambda)>\wt(v)=\mu$, by the selection of $\mu$,
 we must have $\wt(E_{{i_{{s-1}}},h}^{(a_{s-1})} \cdots E_{i_1, h}^{(a_1)}.\mathfrak{m}_\lambda)_h\geq0$. In other words, we have
\begin{equation}\label{(1)}
a_1+\cdots a_{s-1}\leq \lambda_h<a_1+\cdots a_s.
\end{equation}

If, for some $a<b$, $a\neq h$, and $M\in\mathbb Z_{>0}$, $u=E_{a,b}^{(M)}.v\neq 0$, then by \eqref{wei},
$\wt(u)=\mu+M(\epsilon_a-\epsilon_b)>\wt(v)=\mu$ and $\wt(u)_h\leq\wt(v)_h<0$,  contrary to the selection of $\mu$. Thus, we have
\begin{equation}\label{mvan}
E_{a,b}^{(M)}.v= 0 \text{ for all  } 1\leq a< b\leq m+n, a\neq h, M\in\mathbb Z_{>0}.
\end{equation}

\noindent
{\bf Claim 2.} For the selected $v$ as in \eqref{maxmu}, we have
\begin{equation}\label{(2)}
 \prod_{t=1}^{{s}}
                          {\lambda_{h}-(-1)^{{\bar h}+{\bar i_t}}\lambda_{i_t}-a_{t-1}\cdots-a_1 \brack a_t}_q\neq 0.
\end{equation}
{\bf Proof of Claim 2:}
Since $\lambda$ is the highest weight of $L(\lambda)=U_{q,F}.v$ and $\fkm_\lambda\in { U}_{q,F}. v
={ U}_{q,F}^-{ U}_{q,F}^0{ U}_{q,F}^+. v$, no vectors with weight $\lambda$ can occur in the set
$I^-_F{ U}_{F}^0{ U}_{F}^+. v$, where $I^-_F$ is the ideal spanned by all monomials of positive degree.
Hence, we must have
\begin{equation}\notag
\fkm_\lambda\in U_{q,F}.v=
{ U}_{q,F}^+.E_{{i_{{s}}},h}^{(a_s)} \cdots E_{i_1, h}^{(a_1)}\fkm_\lambda.
\end{equation}
By using a PBW type basis for $U_{q,F}^+$ over an ordering on positive roots, 
beginning with  $\alpha_{h,m+n},\alpha_{h,m+n-1},\ldots,\alpha_{h,h+1}$, \eqref{mvan} implies
\begin{equation}\notag
\begin{aligned}
  { U}_{F}^+. v
 &=\span\bigg\{
 \prod_{h+1\leq b\leq m+n}
 E_{h,b}^{(A_{h,b})}\,
\prod_{{\epsilon_a-\epsilon_b\in\Phi^+}\atop a\neq h}
E_{a,b}^{(A_{a,b})}\, .v\;\bigg| \;A\in P(m|n)\bigg\} \\
& =\span\bigg\{
\bigg(\prod_{h+1\leq b\leq m+n}
 E_{h,b}^{(A_{h,b})}\bigg) \;v\bigg|\; A_{h,b}\in\mathbb N\bigg\}.
\end{aligned}
\end{equation}
Thus, 
\begin{equation}\notag
\aligned
({ U}_{F}^+. v)_\lambda&=\span\left\{ 
\prod_{h+1\leq b\leq m+n}
 E_{h,b}^{(A_{h,b})}\,
.v\;\bigg|\;\sum_{b=h+1}^{m+n}A_{h,b}(\epsilon_h-\epsilon_b)=\sum_{t=1}^sa_t(\epsilon_h-\epsilon_{i_t})\right\}\\
&=
\span\left\{ (E_{h,i_1}^{(a_1)}\cdots  E_{h,{i_{{s}}}}^{(a_s)}  )
.(E_{{i_{{s}}},h}^{(a_s)} \cdots E_{i_1, h}^{(a_1)}.\fkm_\lambda)\right\}.
\endaligned
\end{equation}
However, by Lemma \ref{comp}(2),
\begin{equation}\notag
\begin{aligned}
 (E_{h,i_1}^{(a_1)}\cdots  E_{h,{i_{{s}}}}^{(a_s)}  )
.(E_{{i_{{s}}},h}^{(a_s)} \cdots E_{i_1, h}^{(a_1)}.\fkm_\lambda)
=\prod_{t=1}^{{s}}
 {\lambda_{h}-(-1)^{{\bar h}+{\bar i_t}}\lambda_{i_t}-a_{t-1}\cdots-a_1 \brack a_t}_q
                          \fkm_\lambda.\notag
\end{aligned}
\end{equation}
We must have
$ \prod_{t=1}^{{s}}
 {\lambda_{h}-(-1)^{{\bar h}+{\bar i_t}}\lambda_{i_t}-a_{t-1}\cdots-a_1 \brack a_t}_q\neq 0,$
 proving Claim 2.


Now, by the claim \eqref{(2)}, we see
$ {\lambda_{h}-(-1)^{{\bar h}+{\bar i_s}}\lambda_{i_s}-a_{s-1}\cdots-a_1 \brack a_s}_q\neq 0.$
This implies  
$$ {\lambda_{h}-(-1)^{{\bar h}+{\bar i_s}}\lambda_{i_s}-a_{s-1}\cdots-a_1 \geq a_s},$$
or
 $ {\lambda_{h}-(-1)^{{\bar h}+{\bar i_s}}\lambda_{i_s}\geq a_s+a_{s-1}\cdots+a_1  }.$
Thus, the second inequality in \eqref{(1)} forces ${{\bar h}+{\bar i_s}}=1.$ Since $h<i_s$, we must have
 $h\leq m< i_s.$ Hence, $\alpha_{h,i_s}$ is an odd root and so $a_s=1$.
  By \eqref{(1)}, $\la_h=a_1+\cdots+a_{s-1}$ and, consequently,
 $\mu_h=\wt(v)_h=-1$.

Finally, we are ready to prove $j_l(\lambda^{(1)})\geq \lambda_{m}+1.$ Let $s'$ be the minimal index such that
$m<i_{s'}.$ Then $1\leq h<i_1<\cdots<i_{s'-1}\leq m<i_{s'}<i_{s'+1}<\cdots<i_s$. 
This implies $a_{i_t}=1$ for all $s'\leq t\leq s$ and so \eqref{maxmu} becomes
\begin{equation}\notag
\begin{aligned}
v=E_{{i_{{s}}},h} \cdots E_{i_{s'}, h}\,
E_{{i_{{s'-1}}},h}^{(a_{s'-1})} \cdots E_{i_1, h}^{(a_1)}.\fkm_\lambda\;\, \text{ and }\,\;{{\bar h}+{\bar i_t}}=\begin{cases}
1,&s'\leq t\leq s;\\0,&1\leq t< s'.\end{cases}
\end{aligned}
\end{equation}
Since $\wt(v)_h=-1$, it follows that
$s-s'=\lambda_{h}-a_{  i_{s'-1} } -\cdots-a_1 .$
In this case the expression (\ref{(2)}) has the form
\begin{equation}\label{texp}
\begin{aligned}
\prod_{t=s'}^{{s}}
 {(s-s')+\lambda_{i_t}-(t-s') \brack 1}_q
 \prod_{t=1}^{{s'-1}}
 {\lambda_{h}-\lambda_{i_t}-a_{t-1}\cdots -a_1 \brack a_t}_q\neq 0.
\end{aligned}
\end{equation}
The factor for $t=s'-1$ in the second product of (\ref{texp})  being nonzero implies
 $$ {\lambda_{h}-\lambda_{i_{s'-1}}- a_{s'-2}-\cdots -a_1\geq  a_{s'-1}}$$
or equivalently, $s-s'\geq \lambda_{i_{s'-1}} .$

On the other hand, the first product  in (\ref{texp}) can be rewritten as
\begin{equation}\notag
\begin{aligned}
\prod_{t=s'}^{{s}}
 {(s-s')+\lambda_{i_{t}}-(t-s') \brack 1}_q
=\prod_{t=1}^{{s-s'+1}}
 {\lambda_{i_{s'+t-1}}+(s-s'+1)-t \brack 1}_q
\neq 0,
\end{aligned}
\end{equation}
which implies $ j_l(\lambda_{i_{s'}} ,\cdots,\lambda_{i_{s}} )=s-s'+1 $
by Lemma \ref{mull}. Hence, by Lemma \ref{mull} again,
 $$j_l(\lambda^{(1)})\geq j_l(\lambda_{i_{s'}} ,\cdots,\lambda_{i_{s}})
 = s-s'+1\geq \lambda_{i_{s'-1}}+1
 \geq \lambda_{m}+1,$$
noting $i_{s'-1}\leq m$.
Hence, $\lambda\not\in \Lambda^\pp_l(m|n,r)$, as required.
\end{proof}

\section{Classification of irreducible supermodules of $ { S}_{q,F}(m|n,r)$}
We keep the assumption on $F$ and $q$ and assume $l$ is the order of $q^2$ as in \S6. 

\begin{theorem}\label{classification} The set
$\{L(\lambda)
\,
|
\,
\lambda\in \Lambda^\pp_l(m|n,r)\}
$
forms a compete set of all non-isomorphic
irreducible  $S_{q,F}(m|n,r)$-supermodules.
\end{theorem}

\begin{proof} By Proposition \ref{cpol} and Theorem \ref{cirr}, the irreducible supermodules in the set are all irreducible 
$S_{q,F}(m|n,r)$-supermodules. Since every irreducible $S_{q,F}(m|n,r)$-supermodule $L$ is naturally a polynomial irreducible $U_{q,F}$-supermodule of degree $r$ by inflation, it must be of the form $L\cong L(\la)$ by Proposition \ref{ciir}. Now apply Theorem \ref{cirr} to see $\la\in \Lambda^\pp_l(m|n,r)$.
 \end{proof}
 
 \begin{remark} (1)
When $m+n\geq r$, a classification is given in \cite{dgw1, dgw2} without using representations of the quantum supergroup. See also a comparison of the index sets in  \cite[Theorem B.3]{dgw1} in this case. Note that the theorem above has also generalised the classification loc. cit. to the $m+n<r$ case.

(2) The theorem above is a quantum version of \cite[Lemma 5.4]{bru}.

\end{remark}
\begin{corollary}If $q$ is not a root of unity (i.e., if $l=\infty$), then $S_{q,F}(m|n,r)$ is semisimple with
 irreducible representations labelled by $\Lambda^+(r)_{m|n}$ (see \eqref{spartition}).
\end{corollary}
%


We will construct irreducible ${S}_{q,F}(m|n,r)$-supermodules directly in the category $S_{q,F}$-{\bf mod} of finite dimensional $S_{q,F}$-supermodules. In this category, every module $V$ is a weight module in the sense that $V=\oplus_{\la\in\La(m|n,r)}1_\la V$, where $1_\lambda
=\eta_{r, F}({K\brack\lambda})=\phi_{\la,\la}^1$ are weight idempotents. 
In particular, $S_{q,F}(m|n,r)$ itself has a direct 
sum decomposition into projective modules
\begin{equation}\notag
\begin{aligned}
S_{q,F}(m|n,r)=\oplus_{\lambda\in\Lambda(m|n,r)} S_{q,F}(m|n,r) 1_\lambda.
\end{aligned}
\end{equation}

We define analogously the positive part, negative part and zero part $ { S}_{q,F}^+, { S}_{q,F}^-,{ S}_{q,F}^0$ 
for ${ S}_{q,F}(m|n,r)$  which are generated, respectively, by 
$\tte^{(M)}_{a,b}, a<b, M\geq 0;  $ $\tte^{(M)}_{a,b}, a>b,M\geq 0; $ 
$ {\ttk_a\brack t},t\geq 0$, $a\in[1,m+n]$. We may also regard them as homomorphic images of $U_{q,F}^+$, $U_{q,F}^-$, $U_{q,F}^0$, respectively. In particular,
these are subsuperslgebras with identity 1.

Let $I^+$
denote the ideal of $ { S}_{q,F}^+ $ generated by all $\tte^{(M)}_{a,b}, a<b, M> 0$ and define,
for $\lambda\in  \Lambda(m|n,r),$
$$V(\lambda)=S_{q,F} 1_\lambda /S_{q,F}I^+1_\lambda.$$
We may also define the notion of highest weight module in this category.
Thus, if $v$ is a highest weight vector of an $S_{q,F}$-module, then $I^+.v=0$. Call a highest weight module $V$ of highest weight $\la$ to be {\it universal} if every highest weight module with highest weight $\la$ is a homomorphic image of $V$  (cf. \cite[Lem. 3.15]{Br}).
%
%
\begin{theorem} The ${ S}_{q,F}(m|n,r)$-supermodule  $V(\lambda)$ is nonzero  if and only if $\lambda\in \Lambda^\pp_l(m|n,r).$ Moreover, every such a $V(\lambda)$ is an indecomposable universal highest weight supermodule and has a  unique irreducible quotient
$L(\lambda)$.
\end{theorem}
\begin{proof}If $V(\la)\neq0$ then, for any 
highest weight supermodule $V$ of highest weight $\la$, choose a maximal vector
   $\fkm_\lambda\in V_\la$.
Define a map $f$ from  the left ideal $S_{q,F} 1_\lambda$ to $V$ by the rule:
$ f(s 1_\lambda)=(s 1_\lambda). \fkm_\lambda.$
Clearly,  $f$ is a (homogeneous) supermodule homomorphism.
Note that $ f(1_\lambda)= 1_\lambda. \fkm_\lambda= \fkm_\lambda$   
and, for all 
$ s\in S_{q,F}, $ we have 
$ s. \fkm_\lambda= s.(1_\lambda. \fkm_\lambda)= (s1_\lambda). \fkm_\lambda= f(s 1_\lambda).$
Hence, $f$ is a surjection. Since
$$f(S_{q,F}I^+ 1_\lambda)=(S_{q,F}I^+ 1_\lambda).\fkm_\lambda
= S_{q,F}I^+ .\fkm_\lambda=0,$$
we see that $S_{q,F}I^+ 1_\lambda\subseteq \ker f$.
Thus, $f$  induces an epimorphism $ {\bar f}: V(\lambda)\rightarrow V$. This proves the universal property.

If $\lambda\in \Lambda^\pp_l(m|n,r)$, the argument above for $V=L(\lambda)$ shows that $L(\lambda)$
is  a  homomorphic image of   $V(\lambda).$ Hence, $V(\lambda)\neq 0 $. Conversely,  if $V(\lambda)\neq 0 $, then $V(\la)$ has an irreducible head of highest weight $\la$ which must be isomorphic to $L(\la)$. Hence, $\la\in \Lambda^\pp_l(m|n,r)$ by Theorem \ref{classification}.
\end{proof}
\begin{remarks}(1)
The supermodules $V(\lambda)$ play the role of Weyl modules. It would be interesting to determine the formal character of $V(\la)$. 

(2) If we order the set $\Lambda^\pp_l(m|n,r)$ as $\la^{(1)},\la^{(2)},\cdots,\la^{(N)}$ such that $\la^{(i)}\leq\la^{(j)}$ implies $i>j$, then we may construct a filtration of ideals
$$0\subseteq S_{q,F}f_1 S_{q,F}\subseteq S_{q,F}f_2 S_{q,F}\subseteq \cdots S_{q,F}f_NS_{q,F}\subseteq S_{q,F},$$
where $f_i=\sum_{j=1}^i1_{\la^{(i)}}$. Note that, if $n=0$, then the sequence is a heredity chain for the quasi-hereditary algebra $S_{q,F}(m,r)$.

\vspace{.1cm}\noindent
{\bf Claim:} $S_{q,F}f_NS_{q,F}= S_{q,F}$. \vspace{-1ex}
\begin{proof} It suffices to prove ${S}_{q,F} 1_\la{S}_{q,F}\subseteq S_{q,F}f_NS_{q,F}$ for all $\la\in\La(m|n,r)$.
We apply induction on the poset structure of $\La(m|n,r)$. There is nothing to prove if $\la\in \Lambda_l^{++}(m|n,r)$. In particular, the assertion is true for largest element $(r,0,\ldots,0|0,\ldots,0)$. 
Suppose $ \lambda\in \Lambda(m|n,r)/\Lambda_l^{++}(m|n,r).$ Then 
$V(\lambda)={S}_{q,F} 1_\lambda /{S}_{q,F}I^+ 1_\lambda=0.$
So  $1_\lambda\in {S}_{q,F}I^+ 1_\lambda .$
Hence, there exist $x_i\in{S}_{q,F},y_i\in I^+$ with
$1_\lambda=\sum_ix_iy_i1_\lambda.$
But, by Lemma \ref{reab}, $x_iy_i1_\lambda=x_i1_{\lambda^{(i)}} y_i, $ where $\lambda^{(i)}\in  \Lambda(m|n,r),$  and ${\lambda^{(i)}} >{\lambda}.$ 
Hence, by induction,  $ {S}_{q,F} 1_\lambda  {S}_{q,F} \subseteq   \sum_i{S}_{q,F}1_{\lambda^{(i)}}  {S}_{q,F}\subseteq S_{q,F}f_NS_{q,F}.$
\end{proof}

(3) In \cite{dgw1, dgw2}, a classification is done by using the defect groups of primitive idempotents. By (2), we see that the non-equivalent primitive idempotents $e_1,e_2,\cdots,e_N$ can be selected to satisfy the condition $e_i1_{\la^{(i)}}=e_i$ for every $i$. It would be interesting to know if this condition can be used to determine the defect group of $e_i$.
\end{remarks}

%
%
%
\section{The Mullineux map and Serganova's algorithm}

%
%
%
In this section, we keep the notations $F, q,l$ as defined at the beginning of \S6.

Recall the $R$-algebra automorphism $(\;\;)^\sharp$ defined in \eqref{sharpH}. For any ${H}_{q^2,F}$-module $W$,
define a new ${H}_{q^2,F}$-module $W^\sharp$ by twisting the action via $(\;\;)^\sharp$:
$h.v:=h^\sharp\, v.$ Note that, if $q^2= 1$, i.e., if ${H}_{q^2,F}=F\fS_r$ is the group algebra, then $W^\sharp\cong  W\otimes \sgn$.

%
%
%
Let $\La^+_l(r) $ be the set of all $l$-restricted partitions of $r$. Then, by \cite[\S\S4,6]{DJ86}, this set indexes the isomorphism classes of irreducible ${H}_{q^2,F}$-module. Let $D_\lambda,\lambda\in\La_l^+(r) $ denote a representative from the class $\la$; see \eqref{schf} below for the definition of $D_\la$. Thus, $\{D_\lambda\}_{\lambda\in\La_l^+(r)} $ forms a complete set of pairwise non-isomorphic irreducible ${H}_{q^2,F}$-modules.

Define  the Mullineux  conjugation map 
\begin{equation}\label{Mmap}
{\ttM}: \La^+_l(r) \longrightarrow \La_l^+(r),\quad\la\longmapsto \ttM(\la)
\end{equation}
by mimicking the definition on \cite[p.32]{bru} with prime $p$ replaced by $l$ (see also remarks in the second paragraph on \cite[p.556]{Br}). We omit the details here. Note that this is the transpose of the original definition from \cite{M}.

\begin{remarks}\label{Xu} (1) Irreducible $p$-modular representations of the symmetric group $\mathfrak S_r$ are indexed by $p$-regular partitions of $r$. Mullineux conjectured $(D_\lambda)^\sharp\cong
 D_{\ttM(\lambda)}$. This conjecture was first proved by Ford and Kleshchev \cite{FK} building on \cite{K}. Using representations of supergroups, Brundan and Kujawa \cite{bru} gave an excellent new proof for the original conjecture. See also \cite{SW} for the ortho-symplectic super case.

The quantum version of this conjecture was first proved by Brundan \cite{Br}. The main method used there is the Branching Rule. However, it would be interesting to seek a proof of using quantum supergroups and $q$-Schur superalgebras, generalising the idea in \cite{bru} to the quantum case. In the next two sections, we will use the techniques developed in the paper to prove the quantum version of the Mullineux conjecture.

(2) There is another algorithm due to Xu \cite{xu} which is also independent of the primality of $p$. Thus, \cite[Thm 6.1]{bru} continue to hold for all $l>0$.
\end{remarks}


%
%
%

Recall the Lusztig $\mathcal Z$-form $U_{\up,\mathcal Z}(m|n)$, the isomorphism $\sigma$ considered in \eqref{sigma}, its specialisation $\sigma_F$ on
$ U_{q,F}=U_{\up,\mathcal Z}({n|m})\otimes_{\mathcal Z}F$, and the super dot product
$(\;\,,\,\;)_s$ on $\mathbb Z^{m+n}$ introduced at the end of the introduction.
We first generalise Serganova's algorithm for the supergroup $GL(m|n)$ given in \cite[Lem.~4.2, Thm 4.3]{bru})
 to the quantum hyperalgebra $U_{q,F}$. 
 
\begin{proposition}\label{lowe}
Let $\lambda\in \mathbb Z^{m|n}_\pp$ and choose a nonzero vector
$\mathfrak{m}_\lambda\in L(\lambda)_\lambda$ for the irreducible
$U_{q,F}$-module $L(\lambda)$. Fix the following ordering on positive odd roots:
$$\beta_1=\alpha_{m,m+1},\cdots, \beta_m=\alpha_{1,m+1}, \beta_{m+1}=\alpha_{m,m+2},\cdots,\beta_{2m}=\alpha_{1,m+2},\cdots, \beta_{mn}=\alpha_{1,m+n}.$$ 
Define recursively $\mathfrak{m}_\lambda^{(0)}=\mathfrak{m}_\lambda$ and, for $1\leq k\leq m+n$,
$$
\mathfrak{m}_\lambda^{(k)} = \left\{ \begin{array}{lll}
             \mathfrak{m}_\lambda^{(k-1)},  & \mbox{  if  }
                       l\mid ( \wt(\mathfrak{m}_\lambda^{(k-1)}),\beta_k)_s,\\
              E_{-\beta_k}\mathfrak{m}_\lambda^{(k-1)}, &   \mbox{  if  }l\nmid
               ( \wt(\mathfrak{m}_\lambda^{(k-1)}),\beta_k)_s.
           \end{array}
         \right.
$$
Then we have
\begin{equation}\label{maxa}
\begin{cases}
(1)&\mathfrak{m}_\lambda^{(k)}\neq 0,\quad 0\leq k\leq m+n\\
(2)&E_{i,i+1}^{(M)}\, \mathfrak{m}_\lambda^{(k)}=0 ,\quad  1\leq i\leq m+n-1,\;  i \neq m,\\
(3)&E_{-\beta_i}\, \mathfrak{m}_\lambda^{(k)}=0=E_{\beta_j}\, \mathfrak{m}_\lambda^{(k)},   \quad 1\leq i\leq k< j\leq mn.
\end{cases}
\end{equation}
\end{proposition}
\begin{proof}
We apply induction on $k.$ The case for $k=0$ is clear since
$\mathfrak{m}_\lambda^{(0)}=\mathfrak{m}_\lambda$ is a highest weight vector. Assume now $k\geq1$ and that (1)--(3) hold for $k-1.$

{\bf Case 1.} Assume $l\nmid( \wt(\mathfrak{m}_\lambda^{(k-1)}),\beta_k)_s$. Then $\mathfrak{m}_\lambda^{(k)}=
E_{-\beta_k}\mathfrak{m}_\lambda^{(k-1)}$ and $E_{\beta_k}\,\mathfrak{m}_\lambda^{(k-1)}=0$ by induction. Thus, by
Proposition \ref{cqrv}(4)
\begin{equation}\label{posi}
\begin{aligned}
E_{\beta_k}\, \mathfrak{m}_\lambda^{(k)}
 &=E_{\beta_k}\,E_{-\beta_k}\mathfrak{m}_\lambda^{(k-1)}=-E_{-\beta_k}E_{\beta_k}\,\mathfrak{m}_\lambda^{(k-1)}
 +\frac{K_{\beta_k}-K_{\beta_k}^{-1}}{q-q^{-1}}\mathfrak{m}_\lambda^{(k-1)}\\
&=\frac{q^{ ( \wt(\mathfrak{m}_\lambda^{(k-1)}),\beta_k)_s}-q^{ -(
\wt(\mathfrak{m}_\lambda^{(k-1)}),\beta_k)_s}}
 {q-q^{-1}}\mathfrak{m}_\lambda^{(k-1)} \neq 0.
\end{aligned}
\end{equation}
Hence, $\mathfrak{m}_\lambda^{(k)}\neq 0$, proving (1). 

To see (2), we assume $\beta_k=\epsilon_c-\epsilon_d$ with $1\leq c\leq m<d\leq m+n$. If $i+1\leq c$ or $c<i<i+1<d$, then Proposition \ref{cqrv}(1) and induction imply (2);
if $c=i<i+1<d$, then $i=c<m$ and, by Proposition \ref{cqrv}(3), either $ E_{i,i+1} E_{-\beta_k}=xE_{i,i+1}+yE_{-\beta_{k-1}}$ or, for $M>1$,
  $ E_{i,i+1}^{(M)} E_{-\beta_k}=x'E_{i,i+1}^{(M)}+y'E_{i,i+1}^{(M-1)}$. So (2) follows from induction. 

Finally, if $k<j$ and $\beta_k=\epsilon_a-\epsilon_b$, $\beta_j=\epsilon_c-\epsilon_d$, then either $b<d$ or $b=d$,$a>c$. For $b<d$, applying $\Upsilon$ to Proposition \ref{cqrv}(1)(3)(5) (for $b=d$,$a>c$, using directly Proposition \ref{cqrv}(2)) and induction gives $E_{\beta_j}\, \mathfrak{m}_\lambda^{(k)}=0$ in (3). To verify $E_{-\beta_i}\, \mathfrak{m}_\lambda^{(k)}=0$ for all $1\leq i\leq k$, since $E_{-\beta_k}^2=0$, it suffices to consider the
commutator formulas for $E_{-\beta_i} E_{-\beta_k}$ for $1\leq i<k$. Suppose 
$E_{-\beta_{i}}=E_{b,a}, E_{-\beta_{k}}=E_{d,c},$ for some $1\leq a,c\leq m<b,d\leq m+n$. Then $i<k$ implies $ b<d \text{ or } b=d, a>c$.
If $b<d$, then $c<b<d$ and, by applying the automorphism $\varpi$ defined in (\ref{varpi}) to Proposition \ref{ppco}(1)(2)(4), we see that
$E_{-\beta_{i}}E_{-\beta_{k}}
=E_{b,a}E_{d,c}=xE_{b,a}+yE_{b,c}$. 
 If
 $b=d, a>c$ then $c<a<b=d$ and, by applying $\Upsilon$ to Proposition \ref{ppco} (2), we have
$E_{-\beta_{k}}E_{-\beta_{i}}=E_{d,c}E_{b,a}
=-q_cE_{b,a}E_{d,c}=-q_cE_{-\beta_{i}}E_{-\beta_{k}}.$ In both cases,  $E_{-\beta_i}\, \mathfrak{m}_\lambda^{(k)}=0$ follows from induction.

{\bf Case 2.} Assume $l\mid( \wt(\mathfrak{m}_\lambda^{(k-1)}),\beta_k)_s$. Then $\mathfrak{m}_\lambda^{(k)}=
\mathfrak{m}_\lambda^{(k-1)}.$ By induction, it remains to prove $E_{-\beta_k}\mathfrak{m}_\lambda^{(k-1)}=0.$ Suppose
$E_{-\beta_k}\mathfrak{m}_\lambda^{(k-1)}\neq0$. Then $L(\lambda)=U_{q,F}(E_{-\beta_k}\mathfrak{m}_\lambda^{(k-1)})$ and so
$\mathfrak{m}_\lambda^{(k-1)}\in
U_{q,F}(E_{-\beta_k}\mathfrak{m}_\lambda^{(k-1)}).$ 
 
 We claim that $\mathfrak{m}_\lambda^{(k-1)}\in
\span\{E_{\beta_k}E_{-\beta_k}\mathfrak{m}_\lambda^{(k-1)}\}.$ 
Indeed, if $\Phi^+_{\bar0}$ denotes the subset of even roots in $\Phi^+$, by \eqref{ubas} and the commutation formulas of Proposition  \ref{ppco} and
\ref{cqrv}, we see that $U_{q,F}$ is spanned by the elements
\begin{equation*}\label{newb}
\begin{aligned}
\left\{
\prod_{\epsilon_a-\epsilon_b\in\Phi^+_{\bar0}}\!\!E_{b,a}^{(A_{b,a})}
  \prod_{i=k+1}^{mn}E_{-\beta_i}^{\sigma_i} 
 \prod_{i=1}^{k}E_{\beta_i}^{\sigma_i}
\prod_{a=1}^{m+n}\left(K_a^{\delta_a}  {K_a\brack \mu_a}\right)  
 \prod_{i=1}^{k}E_{-\beta_i}^{\sigma_i'} 
 \prod_{i=k+1}^{mn}E_{\beta_i}^{\sigma_i'}
 \prod_{\epsilon_a-\epsilon_b\in\Phi^+_{\bar0}}\!\!E_{a,b}^{(A_{a,b})}
 \right \},\\
\end{aligned}
\end{equation*}where
$ A\in \sM(m|n),\delta_a\in\{0,1\}$ with $\{\sigma_i,\sigma_i'\}=\{A_{\beta_i},A_{-\beta_i}\}$ and $\mu_a=A_{a,a}$. 
By the proof for (2) and (3) above, the elements 
$ \prod_{i=1}^{k}E_{-\beta_i}^{\sigma_i'} 
 \prod_{i=k+1}^{mn}E_{\beta_i}^{\sigma_i'}
 \prod_{\epsilon_a-\epsilon_b\in\Phi^+_0}\!\!E_{a,b}^{(A_{a,b})}
$ vanish $E_{-\beta_k}\mathfrak{m}_\lambda^{(k-1)}$. Thus, we have
\begin{equation}\notag
\begin{aligned}L(\lambda)=\span
\left\{
\prod_{\epsilon_a-\epsilon_b\in\Phi^+_{\bar0}}E_{b,a}^{(A_{b,a})}
  \prod_{i=k+1}^{mn}E_{-\beta_i}^{\sigma_i} 
 \prod_{i=1}^{k}E_{\beta_i}^{\sigma_i}
 .E_{-\beta_k}\mathfrak{m}_\lambda^{(k-1)}\right\} .
\end{aligned}
\end{equation}
We now consider the weight space 
$L(\lambda)_\mu$ with $\mu={\wt(\mathfrak{m}_\lambda^{(k-1)})}$. Since a spanning vector has its weight of the form $\mu+\sum_{\epsilon_a-\epsilon_b\in\Phi^+_0}A_{b,a}(\epsilon_b-\epsilon_a)
-  \sum_{i=k+1}^{mn}{\sigma_i} {\beta_i}
+ \sum_{i=1}^{k}{\sigma_i}{\beta_i}
 -{\beta_k}$, such a vector in $L(\lambda)_\mu$ forces
\begin{equation}\label{eq2}
\begin{aligned}
\sum_{\epsilon_a-\epsilon_b\in\Phi^+_0}A_{b,a}(\epsilon_b-\epsilon_a)
=  \sum_{j=k+1}^{mn}{\sigma_j} {\beta_j}
- \sum_{i=1}^{k}{\sigma_i}{\beta_i}
 +{\beta_k}=:(\nu^{(0)},\nu^{(1)})\in\mathbb Z^{m+n}.
\end{aligned}
\end{equation}
Note the left hand side implies that $|\nu^{(0)}|=|\nu^{(1)}|=0$. This forces $\#X=\#Y$, where
$X=\{i\mid 1\leq i\leq k,\sigma_i\neq0\}$ and $Y=\{j\mid k< j\leq mn,\sigma_j\neq0\}\cup\{k\}$.
Suppose
$ \beta_{i}=\epsilon_{a_i}-\epsilon_{b_i},\beta_{j_i}=\epsilon_{c_i}-\epsilon_{d_i}$, where $i\mapsto j_i$ is a bijection from $X$ to $Y$ and $1\leq a_i,c_i\leq m<b_i,d_i\leq m+n$.
Then  $ \beta_{j_i}-  \beta_{i}=(\epsilon_{c_i}-\epsilon_{a_i})+(\epsilon_{b_i}-\epsilon_{d_i}).$
Thus, $i<j_i$ forces $m<b_i\leq d_i$ and so $\epsilon_{b_i}-\epsilon_{d_i}\in \Phi^+_{\bar0}.$
Hence,  \eqref{eq2} implies
\begin{equation}\notag
\begin{aligned}
\sum_{1\leq a<b\leq m}A_{b,a}(\epsilon_b-\epsilon_a)&=\nu^{(0)}=
\sum_{i\in X} (\epsilon_{c_i}-\epsilon_{a_i}),\\
\sum_{m<a<b\leq m+n}A_{b,a}(\epsilon_{b}-\epsilon_{a})&=\nu^{(1)}=\sum_{i\in X} (\epsilon_{b_i}-\epsilon_{d_i}).\\
\end{aligned}
\end{equation}
The second equality is possible unless both sides are zero. Thus, all $b_i=d_i$, forcing $c_i\leq a_i$. Hence, the first equality must be zero and so $c_i=a_i$ for all $i\in X$. Therefore, we must have
all $ A_{b,a}=0$ and $\sigma_i=\delta_{k,i}.$
Consequently, $ L(\la)_\mu=
\span\{E_{\beta_k}E_{-\beta_k}\mathfrak{m}_\lambda^{(k-1)}\},$ proving the claim. 

Since $ L(\la)_\mu\neq0$, the claim and
(\ref{posi}) force $\frac{q^{ (
\wt(\mathfrak{m}_\lambda^{(k-1)}),\beta_k)_s}-q^{ -(
\wt(\mathfrak{m}_\lambda^{(k-1)}),\beta_k)_s}} {q-q^{-1}}\neq 0$. This implies
$l\nmid( \wt(\mathfrak{m}_\lambda^{(k-1)}),\beta_k)_s$, contrary to the assumption for Case 2.
\end{proof}

Proposition \ref{lowe} gives a (bijecitve) map 
$$\widetilde{\ }:\mathbb Z_\pp^{m|n}\longrightarrow \mathbb Z_\pp^{m|n},\la\longmapsto\widetilde\la:=\wt(\mathfrak{m}_\lambda^{(mn)}),$$
cf.  \cite[(4.1)]{bru}. 
The construction of $\tilde\la$ from $\la$ is known as Serganova's algorithm. 

Assume now $r\leq m,n$. We define  the following two maps as in \cite[\S6]{bru}:
\begin{equation}\label{maxv}
\begin{aligned}
&x: &\La_l^+(r)\rightarrow   \Lambda^\pp(m|n,r),\quad  & \lambda \rightarrow x(\la)=(\lambda,0^{m-r}|0^n));\\
&y: &\La_l^+(r)\rightarrow   \Lambda^\pp(m|n,r),\quad  & \lambda \rightarrow y(\la)=(0^m|\lambda,0^{n-r})).
\end{aligned}
\end{equation}
Serganova's  algorithm together and Xu's algorithm  for the Mullineux map \eqref{Mmap} via the $j_l$ map defined in \eqref{j_l}) have the following relationship as revealed in \cite[Lem. 6.3]{bru}.

\begin{corollary}\label{6.3} If the $m,n\geq r$ and $\la\in\La_l^+(r)$, then $\widetilde{x(\la)}=y(\ttM(\la))$.
\end{corollary}

When $m=n$, we may use this algorithm to compute the highest weight of a simple module twisted by the automorphism $\sigma_F$ on $U_{q,F}({n|n})$; see \eqref{sigR}. 

Recall that, for any $U_{q,F}({n|n})$-supermodule $V$, 
the  $U_{q,F}({n|n})$-supermodule $V^{\sigma}$ is defined by setting $V^\sigma=V$ as a vector space with a new action defined by
$$x\centerdot v=\sigma(x)v, \quad v\in V, \ x\in U_{q,F}({n|n}).$$
The map $V\mapsto V^\sigma$ defines a category isomorphism $U_{q,F}(n|n)$-{\sf mod}$\,\;\cong$ $U_{q,F}(n|n)$-{\sf mod}.
We now use $\tilde\lambda$ to determine the  highest weight
of  the irreducible $U_{q,F}({n|n})$-supermodule $L(\lambda)^{\sigma}$ (cf. \cite[Thm 4.5]{bru}).

\begin{theorem}\label{sigm}
For $\lambda\in \mathbb Z^{n|n}_\pp$, let $L(\la)$ be an irreducible
$U_{q,F}({n|n})$-supermodule with a highest weight vector $\fkm_\la$ and let $\widetilde\lambda=(\tilde\lambda^{(0)}|\tilde\lambda^{(1)})=\wt(\mathfrak{m}_\lambda^{(n^2)})$.
Then the $U_{q,F}({n|n})$-supermodule $L(\lambda)^{\sigma}$ is isomorphic to 
$L(\lambda^{\sigma}),$
where $\lambda^{\sigma}=(\tilde\lambda^{(1)}|\tilde\lambda^{(0)})$.
Furthermore, if we assume $r\leq m=n$, then, for any $\lambda\in\La_l^+(r)$, we have $$L({x(\lambda)})^{\sigma} \cong  L(x(\ttM(\lambda)),$$
where $\ttM $ is the Mullineux map \eqref{Mmap}.
\end{theorem}
\begin{proof}
Since $L(\lambda)^{\sigma}$ is an irreducible supermodule,
it is enough to determine its highest weight.
From the definition of the isomorphism ${\sigma}$,   $v\in L(\lambda)^{\sigma}$
is a maximal vector if and only if    $v$  satisfies:
\begin{equation}\label{maxv}
0=E_{i}^{(M)}\centerdot  v=\sigma(E_{i}^{(M)})\, v=F_{2n-i}^{(M)}\, v ,  
    \quad M>0,\ 1\leq i\leq 2n-1.
\end{equation}
This is equivalent to say that $v$ is a lowest weight vector of $ L(\lambda).$

By Proposition \ref{lowe}, $\mathfrak{m}_\lambda^{(n^2)}$ is a maximal vector for even subhperalgebra $U_{q,F}(n|n)_{\bar0}\cong U_{q,F}(\mathfrak{gl}_n)\otimes U_{q^{-1},F}(\mathfrak{gl}_n)$.
Moreover,
\begin{equation}\label{odro}
\begin{aligned}
E_{-\beta_t}\, \mathfrak{m}_\lambda^{(n^2)}=0,   \ 1\leq t\leq n^2.
\end{aligned}
\end{equation}
Let 
\begin{equation}\notag
\begin{aligned}
\mathfrak{m}^{\sigma}_\lambda=&F_{1}^{(\tilde\lambda_{n-1}^{(0)}-\tilde\lambda_n^{(0)})}
(F_{2}^{(\tilde\lambda_{n-2}^{(0)}-\tilde\lambda_n^{(0)})}
F_{1}^{(\tilde\lambda_{n-2}^{(0)}-\tilde\lambda_{n-1}^{(0)})})
\cdots
(F_{n-1}^{(\tilde\lambda_1^{(0)}-\tilde\lambda_n^{(0)})}\cdots
F_{1}^{(\tilde\lambda_1^{(0)}-\tilde\lambda_2^{(0)})})\\
&\cdot
F_{n+1}^{(\tilde\lambda_{n-1}^{(1)}-\tilde\lambda_n^{(1)})}
(F_{n+2}^{(\tilde\lambda_{n-2}^{(1)}-\tilde\lambda_n^{(1)})}
F_{n+1}^{(\tilde\lambda_{n-2}^{(1)}-\tilde\lambda_{n-1}^{(1)})})
\cdots
(F_{2n-1}^{(\tilde\lambda_1^{(1)}-\tilde\lambda_n^{(1)})}\cdots
F_{n+1}^{(\tilde\lambda_1^{(1)}-\tilde\lambda_2^{(1)})})
\mathfrak{m}^{(n^2)}_\lambda.
\end{aligned}
\end{equation}
Then, by Proposition \ref{lowe2}, we have, for all $ 1\leq i\leq 2n-1$ and $i \neq n$, $F_{i}^{(M)}.\mathfrak{m}_\lambda^{\sigma}=0$.
To see $F_n.\mathfrak{m}^{\sigma}_\lambda=0$, observe the commutation formulas in Proposition \ref{ppco}. If $E_{c,d}$ is odd and $E_{a,b}$ is even, then all the RHS of formulas (1)--(4) are sums of terms starting with an odd root vector. (Only in (4), we need (3) to swap $E_{c,b}^{(t)}E_{c,d}^{(N-t)}$.)
Applying $\Upsilon$ produces half of the required formulas to prove $
F_{n}\, \mathfrak{m}_\lambda^{\sigma}=0.$ The other half can be obtained by applying $\varpi$ to (1)--(4) (see \eqref{fab}), assuming $E_{a,b}$ is odd and $E_{c,d}$ is even. (In (4), we see $E_{a,d}E_{c,b}=E_{c,b}E_{a,b}$ by (1).). Repeatedly applying the eight sets of formulas,we see that  $F_{n}\, \mathfrak{m}_\lambda^{\sigma}=0$ follows from \eqref{odro}.
Hence,  $ \mathfrak{m}_\lambda^{\sigma}$  is a lowest weight vector of  $L(\lambda)$ or a highest weight vector of $L(\lambda)^\sigma$. 

It remains to compute the weight $\wt_{L(\lambda)^\sigma}(\mathfrak{m}_\lambda^{\sigma})$. By Proposition \ref{lowe2}, 
$$
\wt(\mathfrak{m}_\lambda^{\sigma})
=(\tilde\lambda^{(0)}_{m}, \cdots,\tilde\lambda^{(0)}_{1} |\tilde\lambda^{(1)}_{n},\cdots,
                       \tilde\lambda^{(1)}_{1})=\mu
$$
in $L(\lambda).$
Since the isomorphism ${\sigma_F}$ sends $K_i^\pm$  to $K_{2n-i+1}^\mp$ and $K_i\brack t$ to $K_{2n+1-i}^{-1}\brack t$,
it follows from \eqref{wt space}, \eqref{wt space1} that $v\in L(\la)_\mu$ if and only if $v\in (L(\lambda)^{\sigma})_{\mu^\dagger}$. Hence,
$\wt_{L(\lambda)^\sigma}(\mathfrak{m}_\lambda^{\sigma})=\mu^\dagger=(\tilde\lambda^{(1)}|\tilde\lambda^{(0)})=\lambda^{\sigma}$ and, therefore, $\lambda^{\sigma}$
is the highest weight of $L(\lambda)^{\sigma}.$

Now, with  the hypothesis $r\leq m=n$, the last assertion follows from the first assertion and Corollary \ref{6.3}.
\end{proof}
\section{Matching Schur functors and the Mullineux conjecture}
Throughout this section, we assume $m,n\geq r$ and let
$$
\omega=(1^r,0^{m-r}|0^n),
\omega'=(0^m|0^{n-r},1^r)\in \Lambda(m|n,r).$$
We will identify $H_{q^2,F}(r)$ with $1_\omega S_q(m|n,r) 1_{\omega}$ under the isomorphism $t_i\mapsto T_i$, where $t_i$ is defined in the proof of Corollary \ref{qq-1}.

Consider two Schur functors $\fkf_{\omega}, \fkf_{{\omega'}}$ associated with the idempotents $1_\omega,1_{\omega'}$.
Thus, for every $\chi\in\{\omega,\omega'\}$,
$$\fkf_\chi: S_{q,F}(m|n,r)\text{-\sf mod}\longrightarrow 1_\chi S_{q,F}(m|n,r)1_\chi\text{-\sf mod},$$
satisfying $\fkf_\chi(V)=1_\chi V.$
We will make a comparison for the modules $\fkf_\omega(V)$, $\fkf_{\omega'}(V)$.


\begin{lemma} Assume $m\geq r$. If  $\lambda\in\La_l^+(r)$ then $\fkf_{\omega} L(x(\lambda))\neq0$. Hence, $\fkf_{\omega} L(x(\lambda))$ is an irreducible $H_{q^2,F}(r)$-module.
\end{lemma}

\begin{proof} This is clear since $L(x(\lambda))$ contains the irreducible module $L(x(\lambda))_{\bar0}$ for the even quantum subsupergroup $U_{q,F}(n|n)_{\bar0}$
and $1_\omega L(x(\lambda))_{\bar0}\neq0$.
\end{proof}
This lemma guarantees that if we put   (cf. \cite[Thm 5.9, Rem. 5.10]{bru})
\begin{equation}\label{schf}
\begin{aligned}
D_\lambda:=\fkf_{\omega} L(x(\lambda)),
\end{aligned}
\end{equation}
then the set $\{D_\lambda\}_{\lambda\in\La_l^+(r)}$ forms a complete set of irreducible
$H_{q^2,F}(r)$-modules.

The following result follows from Corollary \ref{qq-1}.

\begin{lemma}\label{sche}
Assume $m, n\geq r$ and, for $1\leq i\leq r-1$, let $t_i=q1_\omega \tte_{i}\ttf_{i}1_\omega-1_\omega$ and $t_{m+n-r+i}=q1_{\omega'} \tte_{m+n-r+i}\ttf_{m+n-r+i}
1_{\omega'}-1_{\omega'}.$
Then the map
\begin{equation}\notag
\tau:  {H}_{q^2, F}(r)=1_\omega S_{q,F}(m|n,r)1_\omega  \longrightarrow 1_{\omega'} S_{q,F}(m|n,r)1_{\omega'},\; t_i               \longmapsto t_{m+n-r+i}
\end{equation}
defines an algebra  isomorphism
$1_\omega S_{q}(m|n,r)1_\omega\cong 1_{\omega'} S_{q}(m|n,r)1_{\omega'}$.
\end{lemma}
Thus, we may twist an $1_{\omega'} S_{q,F}(m|n,r)1_{\omega'}$-module $V$ by $\tau$ to get an $H_{q,F}(r)$-module $V^\tau$.
We now establish the relationship between the two Schur functors.
\begin{proposition}\label{difh}
Assume $m,n\geq r$.
For any  $S_{q,F}(m|n,r)$-supermodule $V$, there is
an ${H}_{q^2, F}$-module isomorphism
  $$ (\fkf_{\omega} V)^\sharp\cong (\fkf_{{\omega'}}  V)^\tau. $$
\end{proposition}
\begin{proof}Recall the generators in \eqref{Sgenerators} and let $\tte_{a,b}=\eta_{r,F}(E_{a,b})$ (see \eqref{etaR}).
Let $$\sfF=\tte_{m+n-r+1,1}\tte_{m+n-r+2,2}\cdots \tte_{m+n,r}.$$ 
Then, by Lemma \ref{reab}(6), $\sfF 1_\omega=1_{\omega'} \sfF.$
We first claim that the map
\begin{equation}\notag
g: 1_\omega V  \longrightarrow1_{\omega'} V,\quad 1_\omega v \longmapsto     \sfF 1_\omega v
\;\;(\forall v\in V)\end{equation} 
is a linear isomorphism. Indeed, applying Proposition \ref{cqrv}(4) yields
\begin{equation}\notag
\begin{aligned}
&\quad\,\tte_{r,m+n}\cdots \tte_{2,m+n-r+2}\tte_{1,m+n-r+1}( \sfF1_\omega)\\
&=  \tte_{r,m+n}\cdots \tte_{2,m+n-r+2}(\tte_{1,m+n-r+1} \tte_{m+n-r+1,1})
                      \tte_{m+n-r+2,2}\cdots \tte_{m+n,r}1_\omega\\
&=  \tte_{r,m+n}\cdots \tte_{2,m+n-r+2}{\ttk_{1,m+n-r+1}\brack 1}\tte_{m+n-r+2,2}\cdots \tte_{m+n,r}1_\omega\\
   &\quad\, - \tte_{r,m+n}\cdots \tte_{2,m+n-r+2}(\tte_{m+n-r+1,1}\tte_{1,m+n-r+1})\tte_{m+n-r+2,2}\cdots \tte_{m+n,r}1_\omega\\
&  \overset{(*)}=  \tte_{r,m+n}\cdots \tte_{2,m+n-r+2}\tte_{m+n-r+2,2}\cdots \tte_{m+n,r}1_\omega\\
       &  =\cdots= \tte_{r,m+n} \tte_{m+n,r}1_\omega= 1_\omega.
\end{aligned}
\end{equation}
Here the equality $(*)$ is seen from Lemma \ref{reab}(6) and \eqref{wt space1}, 
since $$\tte_{m+n-r+2,2}\cdots \tte_{m+n,r}1_\omega=1_\lambda \tte_{m+n-r+2,2}\cdots \tte_{m+n,r},$$ where
$\lambda={\omega+\alpha_{m+n-r+2,2}+\cdots+\alpha_{m+n,r}}=(1,0^{m-1}|0^{n-r+1},1^{r-1})$.
Thus, we have $ \sfF1_\omega v\neq 0\iff 1_\omega v\neq 0.$ Hence, $g$ is  injective and so
$\dim  1_\omega V \leq \dim  1_{\omega'} V.$
Similarly, we may use $\sfF'=\tte_{1,m+n-r+1}\cdots\tte_{r-1,m+n-1}\tte_{r,m+n}$ to prove $\dim  1_{\omega'} V \leq \dim  1_{\omega} V$. Hence, $g $ is a bijection.

We now show that the map $g$  is an $H_{q^2,F}$-module isomorphism. This amounts to prove that, for any $v\in  1_\omega V$ and $1\leq i<r$,
\begin{equation}\label{pair1}
g((-t_i+(q^2-1)1_\omega) v)=g(t_i^\sharp v)=\tau(t_i)g( v)=t_{m+n-r+i}g( v).
\end{equation}
We prove \eqref{pair1} by showing that in $S_{q,F}(m|n,r)$
\begin{equation}\label{pair2}
-\sfF t_i1_{\omega}+(q^2-1)\sfF 1_{\omega}= \sfF (t_i^\sharp 1_{\omega})=t_{m+n-r+i} \sfF 1_\omega.
\end{equation}

Let  $n''=m+n-r,$
$\omega''={\omega-\alpha_{r,n''+r}-\cdots-\alpha_{i+2,n''+i+2}}=(1^{i+1},0^{m-i-1}|0^a,1^{r-i-1})$.
Then, for $1\leq i\leq r$,
\begin{equation}\notag
\begin{aligned}
&\quad\;\sfF \cdot 1_\omega \tte_{i}\ttf_{i}1_\omega \\
&=\tte_{m+n-r+1,1}\tte_{m+n-r+2,2}\cdots \tte_{m+n,r} 1_\omega \tte_{i,i+1}\tte_{i+1,i}1_\omega  \\
&= \tte_{n''+1,1}\cdots \tte_{n''+i-1,i-1}\cdot  \tte_{n''+i,i}\tte_{n''+i+1,i+1} \tte_{i,i+1}\tte_{i+1,i}
                       1_{\omega''}\cdot\tte_{n''+i+2,i+2}\cdots
      \tte_{n''+r,r}  1_\omega.
\end{aligned}
\end{equation}
Let (a) stand for Propositions \ref{cqrv}(1); (b) for Propositions \ref{cqrv}(3); (c) for Lemma \ref{reab}(6);
(d) for Lemma \ref{reab}(2). Let (e) be the formula obtained by applying $\Upsilon$ in \eqref{Up} to Propositions \ref{ppco}(3). The middle part of the product above becomes
\begin{equation}\notag
\begin{aligned}
&\quad\;\, \tte_{n''+i,i}(\tte_{n''+i+1,i+1} \tte_{i,i+1})\tte_{i+1,i} 1_{\omega''}\\
&\overset{(a)}=  \tte_{n''+i,i} (\tte_{i,i+1} \tte_{n''+i+1,i+1}) \tte_{i+1,i} 1_{\omega''}
           = (\tte_{n''+i,i} \tte_{i,i+1}) \tte_{n''+i+1,i+1} \tte_{i+1,i} 1_{\omega''}\\
&\overset{(b)}=  (\tte_{i,i+1} \tte_{n''+i,i}+\ttk_{i,i+1} \tte_{n''+i,i+1})  \tte_{n''+i+1,i+1} \tte_{i+1,i} 1_{\omega''} \\
&\overset{(c)}=  \ttk_{i,i+1} \tte_{n''+i,i+1}  \tte_{n''+i+1,i+1} \tte_{i+1,i} 1_{\omega''}
\quad (\text{as }  \tte_{n''+i,i}  \tte_{n''+i+1,i+1} \tte_{i+1,i} 1_{\omega''}=0)\\
&\overset{(d)}=   \tte_{n''+i,i+1}  \tte_{n''+i+1,i+1} \tte_{i+1,i} 1_{\omega''}
          = \tte_{n''+i,i+1}  (\tte_{n''+i+1,i+1} \tte_{i+1,i}) 1_{\omega''}\\
&\overset{(e)}=   \tte_{n''+i,i+1}  (\tte_{n''+i+1,i}
   +q \tte_{i+1,i} \tte_{n''+i+1,i+1}) 1_{\omega''}\\
&\overset{(e)}=   \tte_{n''+i,i+1}  \tte_{n''+i+1,i}  1_{\omega''}
   +q\tte_{n''+i,i}   \tte_{n''+i+1,i+1} 1_{\omega''}+q^2\tte_{i+1,i} \tte_{n''+i,i+1}   \tte_{n''+i+1,i+1} 1_{\omega''} \\
&\overset{(c)}=  \tte_{n''+i,i+1}  \tte_{n''+i+1,i}  1_{\omega''}
   +q\tte_{n''+i,i}   \tte_{n''+i+1,i+1} 1_{\omega''}.
\end{aligned}
\end{equation}
Thus, we have
\begin{equation}\notag
\begin{aligned}
&\quad\;\sfF \cdot1_\omega \tte_{i}\ttf_{i}1_\omega\\
&= \tte_{n''+1,1}\cdots \tte_{n''+i-1,i-1}  \tte_{n''+i,i+1}  \tte_{n''+i+1,i} 1_{\omega''}\tte_{n''+i+2,i+2}     \cdots   \tte_{n''+r,r}  1_\omega\\
      &\quad\; +q \tte_{n''+1,1}\cdots  \tte_{n''+i-1,i-1}   \tte_{n''+i,i}   \tte_{n''+i+1,i+1} 1_{\omega''}\tte_{n''+i+2,i+2}     \cdots   \tte_{n''+r,r}  1_\omega.
\end{aligned}
\end{equation}
Hence,
\begin{equation}\notag
\begin{aligned}
\sfF t_i 1_{\omega} &=\sfF (q1_\omega \tte_{i,i+1}\tte_{i+1,i}1_\omega-1_\omega)\\
&= q\tte_{n''+1,1}\cdots   \tte_{n''+i,i+1}  \tte_{n''+i+1,i}      \cdots   \tte_{n''+r,r}  1_\omega\\
      &\quad\; +(q^2-1) \tte_{n''+1,1}\cdots   \tte_{n''+i,i}   \tte_{n''+i+1,i+1}      \cdots   \tte_{n''+r,r}  1_\omega.\\
 \end{aligned}
\end{equation}    
\begin{equation}
\begin{aligned} 
\sfF t_i^\sharp 1_{\omega} &=\sfF (-t_i 1_{\omega}+(q^2-1) 1_{\omega} )\\
&= -q\tte_{n''+1,1}\cdots \tte_{n''+i-1,i-1} ( \tte_{n''+i,i+1}  \tte_{n''+i+1,i})    \tte_{n''+i+2,i+2}  \cdots   \tte_{n''+r,r}  1_\omega.\label{*}\\
\end{aligned}
\end{equation}

Similarly,
\begin{equation}\notag
\begin{aligned}
&\quad\;1_{\omega'} \tte_{n''+i}\ttf_{n''+i}1_{\omega' }(\sfF 1_{\omega} )\\
&=1_{\omega'} \tte_{n''+i,n''+i+1}\tte_{n''+i+1,n''+i}1_{\omega' }
            \tte_{n''+1,1}\tte_{n''+2,2}\cdots \tte_{n''+r,r} 1_\omega   \\
&= \tte_{n''+1,1}\cdots  \tte_{n''+i-1,i-1}\cdot \tte_{n''+i,n+i+1}\tte_{n''+i+1,n''+i} \tte_{n''+i,i}\tte_{n''+i+1,i+1}
                       1_{\omega''}\\
                       &\quad\; \cdot\tte_{n''+i+2,i+2}\cdots
      \tte_{n''+r,r}  1_\omega,
\end{aligned}
\end{equation}
Let (u) be the formula obtained by applying $\Upsilon$ to Proposition \ref{cqrv}(2) twice; (v) for Lemma \ref{KE};
and (w) for the $\Upsilon$-version of Proposition \ref{ppco}(1). Then the middle part of the product above becomes
\begin{equation}\notag
\begin{aligned}
&\quad\; \,\tte_{n''+i,n''+i+1}(\tte_{n''+i+1,n''+i} \tte_{n''+i,i})\tte_{n''+i+1,i+1}  1_{\omega''}\\
&\overset{(e)}=  \tte_{n''+i,n''+i+1}(\tte_{n''+i+1,i}
     +q^{-1}\tte_{n''+i,i} \tte_{n''+i+1,n''+i})  \tte_{n''+i+1,i+1}  1_{\omega''}\\
&\overset{(c)}=  (\tte_{n''+i,n''+i+1}\tte_{n''+i+1,i}) \tte_{n''+i+1,i+1}  1_{\omega''}\\
&\overset{(u)}= ( \tte_{n''+i+1,i}\tte_{n''+i,n''+i+1} 
+   \tte_{n''+i,i}\ttk_{n''+i,n''+i+1}^{-1}) \tte_{n''+i+1,i+1}  1_{\omega''}\\
&\overset{(v)}= \tte_{n''+i+1,i}(\tte_{n''+i,n''+i+1} \tte_{n''+i+1,i+1} ) 1_{\omega''}
+   q^{-1}\tte_{n''+i,i}\tte_{n''+i+1,i+1} \ttk_{n''+i,n''+i+1}^{-1} 1_{\omega''}\\
&\overset{(u)}=  \tte_{n''+i+1,i} (\tte_{n''+i+1,i+1} \tte_{n''+i,n''+i+1}
+\tte_{n''+i,i+1} \ttk_{n''+i,n''+i+1}^{-1} ) 1_{\omega''}\\
&\overset{\;\;(d)}{\text{\ }}+  q^{-1} \tte_{n''+i,i} \tte_{n''+i+1,i+1}  1_{\omega''}\\
&\overset{(c,d)}=  \tte_{n''+i+1,i}\tte_{n''+i,i+1}  1_{\omega''}
+  q^{-1} \tte_{n''+i,i} \tte_{n''+i+1,i+1}  1_{\omega''}\\
&\overset{(w)}=  -\tte_{n''+i,i+1}\tte_{n''+i+1,i}  1_{\omega''}
+  q^{-1} \tte_{n''+i,i} \tte_{n''+i+1,i+1}  1_{\omega''}.\\
\end{aligned}
\end{equation}
Thus,
\begin{equation}\notag
\begin{aligned}&\quad\;1_{\omega'} \tte_{n''+i}\ttf_{n''+i}1_{\omega' }(\sfF 1_{\omega} )\\
&= -\tte_{n''+1,1}\cdots  \tte_{n''+i-1,i-1}\cdot \tte_{n''+i,i+1}  \tte_{n''+i+1,i}   \cdot\tte_{n''+i+2,i+2}    \cdots   \tte_{n''+r,r}  1_\omega+q^{-1} \sfF  1_\omega.
\end{aligned}
\end{equation}
Hence, by \eqref{*},
\begin{equation}\notag
\begin{aligned}
t_{n''+i}\sfF 1_{\omega} 
&=(q1_{\omega'} \tte_{n''+i}\ttf_{n''+i}1_{\omega' }-1_{\omega'} )(\sfF 1_{\omega} )\\
&= -q\tte_{n''+1,1}\cdots \tte_{n''+i-1,i-1}  (\tte_{n''+i,i+1}  \tte_{n''+i+1,i}) \tte_{n''+i+2,i+2}      \cdots   \tte_{n''+r,r}  1_\omega\\
&=\sfF (t_i^\sharp 1_{\omega} ),
\end{aligned}
\end{equation}
proving \eqref{pair2}, and hence, \eqref{pair1}.
\end{proof}
When $m=n\geq r$, the automorphism ${\sigma_F}$ on $S_{q,F}(n|n,r)$ (see Lemma \ref{sigS}) takes $1_\omega$ to $1_{\omega'}$. So restriction induces an algebra isomorphism (see \eqref{Bsig})
$$\bar\sigma: H_{q^2,F}(r)=1_\omega S_{q,F}(n|n,r)1_\omega\longrightarrow 1_{\omega'} S_{q,F}(n|n,r)1_{\omega'},\quad t_i\longmapsto t_{2n-i},$$
for $1\leq i\leq r-1$. Thus, twisting module actions defines a functor
$$\bar\sigma:1_{\omega'}S_{q,F}(n|n,r)1_{\omega'}{\text{-\sf mod}}\longrightarrow H_{q^2,F}(r){\text{-\sf mod}}.$$
Note that, if $V$ is an $S_{q}(n|n,r)$-supermodule, then $ ( \fkf_{{\omega'}}  V)^{\bar\sigma} $ is an
$H_{q^2,F}(r)$-module via the action $T_i \cdot x=\bar\sigma(t_{i})x$ for all $x\in \fkf_{\omega'}V$. Likewise,  $\fkf_{\omega}(V^{\sigma})$ is an $H_{q^2,F}(r)$-module via the action $T_i \cdot y=\sigma(t_i)y$ for all $y\in \fkf_{\omega} (V^{\sigma})$.
\begin{lemma}\label{schs}
With the notation above,
the following diagram 
$$\begin{CD}
S_{q,F}(n|n,r){\text{-\sf mod}} @>\fkf_\omega>>  H_{q^2,F}{\text{-\sf mod}}  \\ 
@A\sigma AA  @AA\bar\sigma A  \\
S_{q,F}(n|n,r){\text{-\sf mod}} @>\fkf_{\omega'}>> 1_{\omega'}S_{q,F}(n|n,r)1_{\omega'}{\text{-\sf mod}}
\end{CD}$$ is commutative.
In other words, $ ( \fkf_{{\omega'}}  V)^{\bar\sigma} =\fkf_{\omega} (V^{\sigma})$
for any  $S_{q,F}$-supermodule $V$.
\end{lemma}
\begin{proof}
Since $\sigma(1_{\omega'}) =1_{\omega}$, we have $ {1_{\omega'}}  V = {1_\omega}(V^{\sigma})$ or $\fkf_{{\omega'}}  (V)= \fkf_{\omega} (V^{\sigma})$
as vectors spaces. Now it is easy to see from the above that the $H_{q^2,F}$-module structures on both side are the same.
\end{proof}

We are now ready to proof the quantum version of the Mullineux conjecture.
\begin{theorem}\label{QMC}
For any $\lambda\in\La_l^+(r) $, the irreducible ${H}_{q^2,F}(r)$-modules
$D_\lambda^\sharp$ and $ D_{\ttM(\lambda)}$ are isomorphic: $D_\lambda^\sharp\cong D_{\ttM(\lambda)}$.
\end{theorem}
\begin{proof}By definition,
$D_\lambda=\fkf_{\omega}L(x(\lambda)).$
 Then by Proposition \ref{difh}, 
$$(\fkf_{{\omega'}} L(x(\lambda)))^\tau\cong(\fkf_{\omega}L(x(\lambda)))^\sharp
=D_\lambda^\sharp.$$
By Lemma \ref{schs} and Theorem \ref{sigm} then
$$
(\fkf_{{\omega'}} L(x(\lambda)))^{\bar\sigma} \cong \fkf_{{\omega}} (L(x(\lambda))^{\sigma})
\cong \fkf_{{\omega}} L(x(\ttM(\lambda)))=D_{\ttM(\lambda)}.
$$
Since $\tau^{-1}\bar\sigma(t_i)=t_{r-i}$ is the automorphism induced by the graph automorphism
for the Hecke algebra ${H}_{q^2,F}(r)$, we have $(D_\lambda)^{\tau^{-1}\bar\sigma}  \cong  D_\lambda.$
Therefore, 
$$D_\lambda^\sharp\cong
(\fkf_{{\omega'}} L(x(\lambda)))^{\tau} \cong  
((\fkf_{{\omega'}} L(x(\lambda)))^{\tau})^{\tau^{-1}\bar\sigma} \cong  
(\fkf_{{\omega'}} L(x(\lambda)))^{\bar\sigma}\cong  D_{\ttM(\lambda)},
$$
as desired.
\end{proof}

%

\begin{appendix}
\section{The lowest weight of an irreducible $U_{q,F}(m|0)$-module}

Let $U_{q,F}({m})=U_{q,F}({m|0})$ be the quantum hyperalgebra of $\mathfrak{gl}_{m}$ and, for $\lambda\in \mathbb Z^m_+:=\mathbb Z^{m|0}_\pp$, let $L(\la)$ be the associated irreducible $U_{q,F}({m})$-module.

The following result should be the special case of a general result. For example, by the symmetries acting (or braid group actions) on the Weyl module $V(\la)$ (\cite[Ch. 5]{Lubk}), \cite[Lem. 39.1.2]{Lubk} tells exactly the result. However, for our purpose, one needs to extend these actions to the quantum hyperalgebra $U_{q,F}(m)$ and modules at roots of unity, and establish a result for $L(\lambda)$ similar to \cite[Lem. 39.1.2]{Lubk}. For completeness, we provide below a direct proof for the type $A$ case.

\begin{proposition}\label{lowe2} For $\lambda=(\la_1,\cdots,\lambda_m)\in \mathbb Z^m_+$,
if $0\neq\mathfrak{m}_\lambda\in L(\lambda)_\la$ is a highest weight vector, then 
$$F_{1}^{(\lambda_{m-1}-\lambda_{m})}(F_2^{(\la_{m-2}-\la_m)}F_1^{(\la_{m-2}-\la_{m-1})})
\cdots(F_{m-1}^{(\lambda_{1}-\lambda_m)}\cdots
F_{2}^{(\lambda_{1}-\lambda_{3})}
F_{1}^{(\lambda_{1}-\lambda_{2})})
\mathfrak{m}_\lambda\neq0$$ is a lowest weight vector of $L(\lambda)$ with weight $\lambda^\dagger=(\lambda_m,\la_{m-1},\cdots,\lambda_1)$.
\end{proposition}
\begin{proof}Define recursively
$$
\mathfrak{n}_\lambda^{(k)} =\begin{cases} \mathfrak{m}_\lambda,&\text{ if } k=0;\\
F_{m-k}^{(\lambda_{k}-\lambda_m)}\cdots
F_{2}^{(\lambda_{k}-\lambda_{k+2})}
F_{1}^{(\lambda_{k}-\lambda_{k+1})}
\mathfrak{n}_\lambda^{(k-1)},&\text{ if }1\leq k\leq m-1.\end{cases}
$$
We first claim that, for all $0\leq k\leq m-1$
\begin{equation}\label{maxa2}
\begin{aligned}
(1)&\quad\mathfrak{n}_\lambda^{(k)}\neq0;\\
(2)&\quad\wt({\mathfrak{n}_\lambda^{(k)}})=(\lambda_{k+1},\lambda_{k+2},\cdots,\lambda_m,
\lambda_{k},\lambda_{k-1},\cdots,\lambda_1),\\
(3)&\quad E_{i}^{(M)}\, \mathfrak{n}_\lambda^{(k)}=0 , \mbox{ for } i<m-k,  M>0.
\end{aligned}
\end{equation}
Indeed, it is obvious if $k=0$. Assume now $k>0$ and that (1)--(3) hold for $k-1.$
Then ${\mathfrak{n}_\lambda^{(k-1)}}\neq0$ and, by Proposition \ref{cqrv}(4)(1),
$$\aligned
E_{1}^{(\lambda_{k}-\lambda_{k+1})}&\cdots
E_{m-k-1}^{(\lambda_{k}-\lambda_{m-1})}
E_{m-k}^{(\lambda_{k}-\lambda_m)}
{\mathfrak{n}_\lambda^{(k)}}\\
&=E_{1}^{(\lambda_{k}-\lambda_{k+1})}\cdots
E_{m-k-1}^{(\lambda_{k}-\lambda_{m-1})}
E_{m-k}^{(\lambda_{k}-\lambda_m)}F_{m-k}^{(\lambda_{k}-\lambda_m)}
F_{m-k-1}^{(\lambda_{k}-\lambda_{m-1})}\cdots
F_{1}^{(\lambda_{k}-\lambda_{k+1})}
{\mathfrak{n}_\lambda^{(k-1)}}\\
&=E_{1}^{(\lambda_{k}-\lambda_{k+1})}\cdots
E_{m-k-1}^{(\lambda_{k}-\lambda_{m-1})}
 \begin{bmatrix} \tilde K_{m-k}; 0 \\\lambda_{k}-\lambda_m  \end{bmatrix}
F_{m-k-1}^{(\lambda_{k}-\lambda_{m-1})}\cdots
F_{1}^{(\lambda_{k}-\lambda_{k+1})}
{\mathfrak{n}_\lambda^{(k-1)}}.\\
\endaligned$$
If $\mu=\wt(F_{m-k-1}^{(\lambda_{k}-\lambda_{m-1})}\cdots
F_{1}^{(\lambda_{k}-\lambda_{k+1})}
{\mathfrak{n}_\lambda^{(k-1)}})$, then $\mu_{m-k}=\la_{m-1}+(k-\la_{m-1})=k$ and $\mu_{m-k+1}=\la_m$.
Thus, by \eqref{wt space} and induction, we have
$$\aligned
E_{1}^{(\lambda_{k}-\lambda_{k+1})}&\cdots
E_{m-k-1}^{(\lambda_{k}-\lambda_{m-1})}
E_{m-k}^{(\lambda_{k}-\lambda_m)}
{\mathfrak{n}_\lambda^{(k)}}\\
&=E_{1}^{(\lambda_{k}-\lambda_{k+1})}\cdots
E_{m-k-1}^{(\lambda_{k}-\lambda_{m-1})}
F_{m-k-1}^{(\lambda_{k}-\lambda_{m-1})}\cdots
F_{1}^{(\lambda_{k}-\lambda_{k+1})}
{\mathfrak{n}_\lambda^{(k-1)}}={\mathfrak{n}_\lambda^{(k-1)}}\neq0,
\endaligned$$
proving ${\mathfrak{n}_\lambda^{(k)}}\neq0$.
Also, by induction
\begin{equation}\notag
\begin{aligned}
&\wt({\mathfrak{n}_\lambda^{(k)}})=
\wt(F_{m-k}^{(\lambda_{k}-\lambda_m)}\cdots
F_{2}^{(\lambda_{k}-\lambda_{k+2})}
F_{1}^{(\lambda_{k}-\lambda_{k+1})}
\mathfrak{n}_\lambda^{(k-1)})\\
&=(\lambda_{k},\lambda_{k+1},\cdots,\lambda_m,
\lambda_{k-1},\lambda_{k-2},\cdots,\lambda_1)-(\lambda_{k}-\lambda_{k+1})(\epsilon_{1}-\epsilon_{2})\\
&\quad\,
  -(\lambda_{k}-\lambda_{k+2})(\epsilon_{2}-\epsilon_{3})-\cdots-(\lambda_{k}-\lambda_{m})(\epsilon_{m-k}-\epsilon_{m-k+1})\\
 & =(\lambda_{k+1},\lambda_{k+2},\cdots,\lambda_m,
\lambda_{k},\lambda_{k-1},\cdots,\lambda_1),
\end{aligned}
\end{equation}
proving (2). It remains to prove (3). 

For $i<m-k$ and $M>0$,  let $s=\min(M,(\lambda_{k}-\lambda_{k+i}))$. By Proposition \ref{cqrv}(1)(4),
\begin{equation}\label{xxx}
\begin{aligned}
&\quad\,E_{i}^{(M)}\, \mathfrak{n}_\lambda^{(k)}
=E_{i}^{(M)} (F_{m-k}^{(\lambda_{k}-\lambda_m)}\cdots
F_{2}^{(\lambda_{k}-\lambda_{k+2})}
F_{1}^{(\lambda_{k}-\lambda_{k+1})}
\mathfrak{n}_\lambda^{(k-1)})\\
&=       F_{m-k}^{(\lambda_{k}-\lambda_m)}\cdots
F_{i+1}^{(\lambda_{k}-\lambda_{k+i+1})}
(E_{i,}^{(M)}\cdot
F_{i}^{(\lambda_{k}-\lambda_{k+i})})\cdots
F_{2}^{(\lambda_{k}-\lambda_{k+2})}
F_{1}^{(\lambda_{k}-\lambda_{k+1})}
\mathfrak{n}_\lambda^{(k-1)}  \\
&=F_{m-k}^{(\lambda_{k}-\lambda_m)}\cdots F_{i+1}^{(\lambda_{k}-\lambda_{k+i+1})}
\sum_{t=0}^{s}
       F_{i}^{(\lambda_{k}-\lambda_{k+i}-t)}
       \begin{bmatrix} K_{i,i+1}; 2t-M-(\lambda_{k}-\lambda_{k+i}) \\ t \end{bmatrix} E_{i}^{(M-t)}\\
 &\quad \cdot F_{i-1}^{(\lambda_{k}-\lambda_{k+i-1})}    \cdots
F_{2}^{(\lambda_{k}-\lambda_{k+2})}
F_{1}^{(\lambda_{k}-\lambda_{k+1})}
\mathfrak{n}_\lambda^{(k-1)}.\\
\end{aligned}
\end{equation}
 If $M>\lambda_{k}-\lambda_{k+i}$, then $M-t>0$
for all $0\leq t\leq s=\lambda_{k}-\lambda_{k+i}$ and, by induction,
\begin{equation}\notag
\begin{aligned}
E_{i}^{(M)}\, \mathfrak{n}_\lambda^{(k)}&=\cdots F_{i+1}^{(\lambda_{k}-\lambda_{k+i+1})}
\sum_{t=0}^{s}
       F_{i}^{(\lambda_{k}-\lambda_{k+i}-t)}
       \begin{bmatrix} K_{i,i+1}; 2t-M-(\lambda_{k}-\lambda_{k+i}) \\ t \end{bmatrix} \\
 &\quad \cdot F_{i-1}^{(\lambda_{k}-\lambda_{k+i-1})}     \cdots
F_{2}^{(\lambda_{k}-\lambda_{k+2})}
F_{1}^{(\lambda_{k}-\lambda_{k+1})}
(E_{i}^{(M-t)}
\mathfrak{n}_\lambda^{(k-1)})=0.
\end{aligned}
\end{equation}
If $M\leq (\lambda_{k}-\lambda_{k+i})$, then $s=M$ and, with a similar argument, \eqref{xxx} becomes
\begin{equation}\notag
\begin{aligned}
&\quad\;E_{i}^{(M)}\, \mathfrak{n}_\lambda^{(k)}=F_{m-k}^{(\lambda_{k}-\lambda_m)}\cdots F_{i+1}^{(\lambda_{k}-\lambda_{k+i+1})}
       F_{i}^{(\lambda_{k}-\lambda_{k+i}-M)}
       \begin{bmatrix} K_{i,i+1}; M-(\lambda_{k}-\lambda_{k+i}) \\ M \end{bmatrix}\\
       &\qquad\qquad\qquad\;\cdot F_{i-1}^{(\lambda_{k}-\lambda_{k+i-1})}\cdots
F_{2}^{(\lambda_{k}-\lambda_{k+2})}
F_{1}^{(\lambda_{k}-\lambda_{k+1})}
\mathfrak{n}_\lambda^{(k-1)}\\
&=F_{m-k}^{(\lambda_{k}-\lambda_m)}\!\cdots (F_{i+1}^{(\lambda_{k}-\lambda_{k+i+1})}
       F_{i}^{(\lambda_{k}-\lambda_{k+i}-M)}) \cdot F_{i-1}^{(\lambda_{k}-\lambda_{k+i-1})}\cdots
F_{2}^{(\lambda_{k}-\lambda_{k+2})}
F_{1}^{(\lambda_{k}-\lambda_{k+1})}
\mathfrak{n}_\lambda^{(k-1)}.\\
\end{aligned}
\end{equation}
By applying $\Upsilon$ to Proposition \ref{ppco}(3), we obtain
$$F_{i+1}^{(\lambda_{k}-\lambda_{k+i+1})}
       F_{i}^{(\lambda_{k}-\lambda_{k+i}-M)}=\sum_{t=0}^{s'}q^{a_tb_t}
F_{i}^{(\lambda_{k}-\lambda_{k+i}-M-t)}
E_{i+2,i}^{(t)}F_{i+1}^{(\lambda_{k}-\lambda_{k+i+1}-t)},$$
where $s'=\min((\lambda_{k}-\lambda_{k+i+1}),(\lambda_{k}-\lambda_{k+i}-M))$, $a_t=\lambda_{k}-\lambda_{k+i}-M-t,b_t=
\lambda_{k}-\lambda_{k+i+1}-t$. Substituting gives
$$\aligned
E_{i}^{(M)}\, \mathfrak{n}_\lambda^{(k)}&=\sum_{t=0}^{s'}q^{a_tb_t}
\big(F_{m-k}^{(\lambda_{k}-\lambda_m)}\cdots
F_{i}^{(\lambda_{k}-\lambda_{k+i}-M-t)}
E_{i+2,i}^{(t)}\\
& \quad  \cdot F_{i-1}^{(\lambda_{k}-\lambda_{k+i-1})}    \cdots
F_{2}^{(\lambda_{k}-\lambda_{k+2})}
F_{1}^{(\lambda_{k}-\lambda_{k+1})}
F_{i+1}^{(\lambda_{k}-\lambda_{k+i+1}-t)}
\mathfrak{n}_\lambda^{(k-1)}\big)\\
&=0.
\endaligned$$
Here the last equation follows from induction and \cite[Prop. 6.25]{ddp}. Indeed, by restricting $L(\lambda)$ to the subalgebra $U_{q,F}({m-k+1})$ of $U_{q,F}({m})$ and induction, $\mathfrak{m}_\mu=\mathfrak{n}_\lambda^{(k-1)}$ is a maximal vector of weight $\mu=(\lambda_k,\cdots,\lambda_m)$ with $\mu_i=\la_{k+i-1}$. Since
$$
\lambda_{k}-\lambda_{k+i+1}-t\geq \lambda_{k}-\lambda_{k+i+1}-(\lambda_{k}-\lambda_{k+i}-M)
>\lambda_{k+i}-\lambda_{k+i+1},$$
\cite[Prop. 6.25]{ddp} implies $
F_{i+1}^{(N)}\, \mathfrak{m}_\mu=0 
$ for  all $N>\mu_{i+1}-\mu_{i+2}$.
This completes the proof of (3).

Finally, by the claim, $\lambda^\dagger=\wt(\mathfrak{n}_\lambda^{(m-1)})$ is a weight of $L(\la)$. Since 
$\lambda^\dagger$ is the lowest weight of the Weyl module $V(\lambda)$ and
$L(\lambda)$ is a quotient of $V(\lambda)$, we conclude that
$\lambda^\dagger$ is the lowest weight of $L(\lambda).$
\end{proof}
\end{appendix}

%
%
%

%
%

\makeatletter
\def\@biblabel#1{#1.}
\makeatother

\end{document}